\documentclass[reqno]{amsart}
\usepackage{fullpage,palatino,tikz}
\usepackage[utf8]{inputenc}
\usepackage{color,amsfonts,amsmath,amssymb,amsthm,cases,wasysym,ifthen,bbm}
\usepackage[normalem]{ulem}
\usepackage{tikz-cd}
\usepackage{hyperref}
\hypersetup{colorlinks, citecolor=red, filecolor=black, linkcolor=blue, urlcolor=blue}

\usepackage{enumitem}
\setlist[itemize]{noitemsep}
\setlist[enumerate]{noitemsep}

\numberwithin{equation}{section}

\definecolor{forest}{rgb}{0.16, 0.5, 0.0}
\newcommand{\Red}[1]{{\color{red}{#1}}}

\newcommand{\Blue}[1]{{\color{blue}{#1}}}

\newcommand{\clo}[2][]{%
    \ifthenelse{ \equal{#1}{} }
        {\cl(#2)}
        {\cl_{#1}(#2)}
}

\newcommand{\coveredby}{\lessdot}

\newcommand{\fld}{\Bbbk}

\newcommand{\x}{\times}
\newcommand{\bd}{\partial}
\newcommand{\isom}{\cong}
\newcommand{\0}{\emptyset}
\newcommand{\sm}{\setminus}
\newcommand{\Sym}{\mathfrak{S}}
\newcommand{\st}{\colon} 

\newcommand{\Cc}{\mathbb{C}}
\newcommand{\Ff}{\mathbb{F}}
\newcommand{\Nn}{\mathbb{N}}

\newcommand{\Rr}{\mathbb{R}}

\newcommand{\bb}{\mathbf{b}}
\newcommand{\cc}{\mathbf{c}}
\newcommand{\ee}{\mathbf{e}}

\newcommand{\vv}{\mathbf{v}}

\newcommand{\xx}{\mathbf{x}}
\newcommand{\yy}{\mathbf{y}}

\newcommand{\XX}{\mathcal{X}} 
\newcommand{\YY}{\mathcal{Y}}

\newcommand{\B}{\mathcal{B}}
\newcommand{\D}{\mathcal{D}}
\newcommand{\F}{\mathcal{F}}

\newcommand{\K}{\mathcal{K}}
\renewcommand{\L}{\mathcal{L}}
\newcommand{\N}{\mathcal{N}}
\renewcommand{\P}{\mathcal{P}}

\newcommand{\ccc}{\mathfrak c} 
\newcommand{\ppp}{\mathfrak p}
\newcommand{\qqq}{\mathfrak q}
\newcommand{\sss}{\mathfrak s}
\newcommand{\BBB}{\mathfrak B} 
\newcommand{\AAA}{\mathfrak A} 
\newcommand{\III}{\mathfrak I} 

\DeclareMathOperator{\Atom}{Atom} 
\DeclareMathOperator{\cl}{\mathsf{cl}}  
\DeclareMathOperator{\codim}{codim}
\DeclareMathOperator{\conv}{conv}

\DeclareMathOperator{\Gr}{Gr}

\DeclareMathOperator{\Irr}{Irr}
\DeclareMathOperator{\Le}{Le} 
\DeclareMathOperator{\Ord}{Ord} 

\DeclareMathOperator{\rk}{rk}
\DeclareMathOperator{\rec}{rec} 

\newcommand{\llex}{<_{\text{lex}}}

\newcommand{\lgale}{<_{\text{gale}}}

\newcommand{\legale}{\le_{\text{gale}}}

\newcommand{\defterm}[1]{\textbf{#1}}

\newcommand{\supn}{\sup\nolimits}
\newcommand{\infn}{\inf\nolimits}

\newtheorem{thm}{Theorem}
\numberwithin{thm}{section}

\newtheorem{prop}[thm]{Proposition}
\newtheorem{cor}[thm]{Corollary}

\newtheorem{problem}[thm]{Problem}

\theoremstyle{definition}
\newtheorem{defn}[thm]{Definition}
\newtheorem{example}[thm]{Example}
\newtheorem{remark}[thm]{Remark}

\theoremstyle{remark}
\newtheorem{qn}[thm]{Question}

\author[JB]{Jonah Berggren}
\address{Department of Mathematics, University of Kentucky, Lexington, KY, United States}
\email{jrberggren@uky.edu}
\author[JLM]{Jeremy L.\ Martin}
\address{Department of Mathematics, University of Kansas, Lawrence, KS, United States}
\email{jlmartin@ku.edu}
\author[JAS]{Jos\'e A.\ Samper}
\address{Facultad de Matem\'aticas, Pontificia Universidad Cat\'olica, Santiago, Chile}
\email{jsamper@uc.cl}
\title{Unbounded Matroids}
\date{\today}
\keywords{Matroid, matroid base polytope, polyhedron, submodular system, shellable, subspace arrangement}

\thanks{JLM was supported in part by Simons Foundation Collaboration Grant \#315347 and by an award from the University of Kansas International Affairs Latin America Fund.  JAS was partially supported by FONDECYT Iniciaci\'on grant \#11221076.}
\thanks{Data availability statement: Not applicable to this article as no datasets were generated or analyzed during the current study.}

\subjclass[2020]{
05B35, 
52B40, 
52C35 
}

\begin{document}
\maketitle

\begin{abstract}
A matroid base polytope is a polytope in which each vertex has 0,1 coordinates and each edge is parallel to a difference of two coordinate vectors.  Matroid base polytopes are described combinatorially by integral submodular functions on a boolean lattice, satisfying the unit increase property.  We define a more general class of \textit{unbounded matroids}, or \textit{U-matroids}, by replacing the boolean lattice with an arbitrary distributive lattice.  U-matroids thus serve as a combinatorial model for polyhedra that satisfy the vertex and edge conditions of matroid base polytopes, but may be unbounded.  Like polymatroids, U-matroids generalize matroids and arise as a special case of submodular systems.  We prove that every U-matroid admits a canonical largest extension to a matroid, which we call the \textit{generous extension}; the analogous geometric statement is that every U-matroid base polyhedron contains a unique largest matroid base polytope.  We show that the supports of vertices of a U-matroid base polyhedron span a shellable simplicial complex, and we characterize U-matroid basis systems in terms of shelling orders, generalizing Bj\"orner's and Gale's criteria for a simplicial complex to be a matroid independence complex.  Finally, we present an application of our theory to subspace arrangements and show that the generous extension has a natural geometric interpretation in this setting.
\end{abstract}

\section{Introduction}

A \defterm{matroid} is a combinatorial structure that abstracts the idea of linear independence.  As such, matroids have bases, independent sets, spanning sets, and a wealth of other equivalent combinatorial axiomatizations; the equivalences between these axiomatizations were dubbed ``cryptomorphisms'' by Birkhoff, with the name and idea popularized by Gian-Carlo Rota.  Here we focus on the geometric axiomatization of a matroid as its \defterm{base polytope}, the convex hull of the characteristic vectors of its bases.  A famous theorem of Gel'fand, Goresky, Macpherson, and Serganova \cite{GGMS} states that a polytope $\ppp$ is a matroid base polytope if and only if (a) every vertex has 0,1-coordinates and (b) every edge is parallel to the difference of two standard basis vectors (equivalently, the normal fan of $\ppp$ coarsens the braid fan).  More generally, the polytopes satisfying condition (b), the so-called \defterm{generalized permutahedra} studied by Postnikov \cite{Postnikov}, have received considerable recent attention; see, e.g., \cite{FS,PRW,AA}.  Essentially equivalent but earlier definitions are the \defterm{polymatroids} of Edmonds \cite{Edmonds}, and the \defterm{M-convex sets} of Murota \cite{Murota}.
Even more generally, possibly-unbounded polyhedra satisfying condition (b) (called ``extended generalized permutahedra'' in \cite{AA}) correspond to \defterm{submodular systems}, which have been studied extensively in combinatorial optimization; see \cite{Fujishige}. 

This paper initiates the study of \defterm{unbounded matroid polyhedra}, which are generalized permutahedra $\ppp\subset\Rr^n$ all of whose vertices have 0,1-coordinates, but which need not be bounded.  The recession cone (which describes the unbounded directions) can be recorded by a distributive sublattice $\D\subseteq2^E$, the \defterm{characteristic lattice} of the unbounded matroid.  Specifically, the maximal chains in~$\D$ correspond to the chambers of the braid arrangement containing linear functionals bounded on $\ppp$, so that $\ppp$ is a polytope if and only if $\D=2^E$.  On the combinatorial side, we define an \defterm{unbounded matroid} (or \defterm{U-matroid}) as a triple $U=(E,\D,\rho)$, where $E$ is a finite ground set; $\D\subseteq 2^E$ is a distributive lattice; and $\rho:\D\to\Nn_{\geq0}$ satisfies the usual axioms of a matroid rank function, including submodularity.  By Birkhoff's theorem, $\D=J(\P)$ is the lattice of order ideals of some unique poset $\P$ on $E$, so $U$ may be regarded equivalently as a structure over $\D$ or over $\P$.  When $\D=2^E$, we recover the definition of a matroid.  Thus unbounded matroids and their base polyhedra fit into the framework of submodular systems, joining polymatroids as a special kind of submodular system that includes matroids as a subclass.  The relationships between all these objects, in their combinatorial and geometric incarnations, is shown in Figure~\ref{fig:cube}.

\begin{figure}
\begin{center}
\newcommand{\Ux}{2.5} \newcommand{\Uy}{1.5} 
\newcommand{\Gx}{8.0} \newcommand{\Gy}{0.0} 
\newcommand{\Zx}{0.0} \newcommand{\Zy}{3.0} 
\begin{tikzpicture}
\coordinate (Mat) at (0,0);
\coordinate (UMat) at (\Ux,\Uy);
\coordinate (PM) at (\Zx,\Zy);
\coordinate (SMS) at (\Ux+\Zx,\Uy+\Zy);
\coordinate (MBP) at (\Gx,\Gy);
\coordinate (UMBP) at (\Ux+\Gx,\Uy+\Gy);
\coordinate (GP) at (\Gx+\Zx,\Gy+\Zy);
\coordinate (UGP) at (\Ux+\Gx+\Zx,\Uy+\Gy+\Zy);
\draw (Mat)--(UMat)--(SMS)--(PM)--cycle;
\draw (MBP)--(UMBP)--(UGP)--(GP)--cycle;
\draw (Mat)--(MBP) (UMBP)--(UMat) (SMS)--(UGP) (PM)--(GP);
\node[fill=white] at (Mat) {Matroids};
\node[red,fill=white] at (UMat) {U-matroids};
\node[fill=white] at (PM) {Polymatroids};
\node[fill=white] at (SMS) {Submodular systems};
\node[fill=white, align=center] at (MBP) {Matroid base\\ polytopes};
\node[red,fill=white, align=center] at (UMBP) {Unbounded matroid\\ base polyhedra};
\node[fill=white, align=center] at (GP) {Generalized\\ permutahedra};
\node[fill=white, align=center] at (UGP) {Extended generalized\\ permutahedra};
\end{tikzpicture}
\end{center}
\caption{Matroids, matroid base polytopes, and their generalizations.  Horizontal lines are bijections, the rest inclusions.  The sides of the cube have these meanings: left/right = combinatorial/geometric; bottom/top = integral/real; front/back = bounded/possibly unbounded.\label{fig:cube}}
\end{figure}
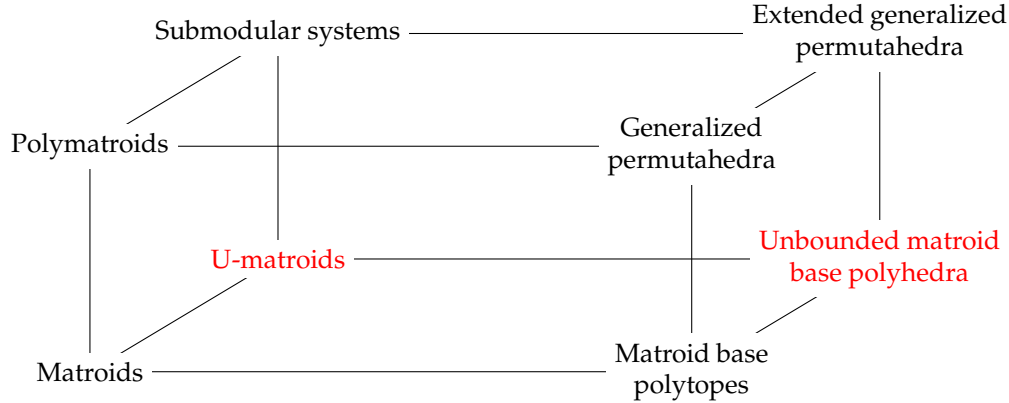

For this reason, the study of unbounded matroids addresses a question proposed by Gian-Carlo Rota: how does one define a matroid over a poset?  Answers were proposed by, among others, Nguyen \cite{Nguyen}, Faigle \cite{Faigle}, Korte--Lov\'asz~\cite{KorteLovasz}, and Barnabei--Nicoletti--Pezzoli \cite{Barnabei}.  In general these authors followed a purely combinatorial approach, focusing on generalizing the core matroid axiomatizations (e.g., bases, independent sets, circuits, rank function, greedy algorithm) and establishing cryptomorphisms between them.  There are often multiple choices for the definitions that are reasonable from a combinatorial viewpoint: for instance, \cite{Faigle} studies two distinct generalizations of matroid independence over a partial order.
For us, polyhedral geometry plays just as important a role as pure combinatorics: the geometry of the base polyhedron points the way to defining the right analogues of matroid axiomatizations such as bases.  Specifically:

\begin{enumerate}
\item Previous authors have generally regarded the partial order $\P$ underlying a generalized matroid as fixed.  By contrast, we want to understand what it means to add or subtract relations from $\P$ (which, here, is the poset of join-irreducibles in the characteristic lattice~$\D$).  This question is natural from a geometric point of view: we will see that containment between U-matroid polyhedra corresponds to restriction and extension of rank functions.  Moreover, we will prove that every U-matroid can be ``generously'' extended to a matroid in a canonical way, corresponding to the unique largest matroid polytope contained in its base polyhedron.  In particular, every unbounded matroid is the restriction of a classical matroid (not necessarily unique) to a poset.

\item We define the bases of a U-matroid as the supports of the vertices of the polyhedron.  The combinatorial definition of bases in terms of the rank function, due to Fujishige in the setting of submodular systems, is more subtle than ``minimal sets of maximal rank''.  One justification for our choice is that the simplicial complex generated by the bases is shellable (and indeed the proof relies crucially on geometry), and we are able to characterize such complexes exactly, as follows.

\item Total orderings on the ground set play a vital role in matroid theory.  By theorems of Bj\"orner \cite[Thm.~7.3.4]{Bjorner-shellability} and Gale~\cite{Gale} respectively, a pure simplicial complex is a matroid complex if and only if (i) for every ordering on the ground set, the corresponding lexicographic order on facets is a shelling order, and if and only if (ii) for every total ordering on the ground set, there is a unique minimal basis with respect to the Gale (or product) order.  We obtain exact analogues of Bj\"orner's and Gale's theorems in which the relevant orderings are the maximal chains in~$\D$.
\end{enumerate}

One application of U-matroids is to linear subspace arrangements~\cite{Bjorner-subspace}.  By way of background, the central combinatorial invariant of an arrangement $\XX$ is its (polymatroid) rank function, which specifies the codimension of every intersection of its spaces.  By work of Goresky and Macpherson~\cite{GM}, this data determines the homology groups of the complement of~$\XX$; it is a major open problem to describe the cohomology ring of the complement combinatorially.  
Barnabei, Nicoletti and Pezzoli~\cite{Barnabei} recorded the rank function of an arrangement by a \defterm{poset matroid}
(which we show is in fact a special kind of U-matroid), and Crowley, Huh, Larson, Simpson and Wang~\cite{Crowley} by a \defterm{multisymmetric matroid}.  Our combinatorial operation of generous extension corresponds to a natural operation on $\XX$: replacing one of its spaces $X$ with generically chosen spaces whose intersection is~$X$.

Unbounded polyhedra have proven to be valuable objects of study even if one's main objects of study are bounded polytopes.  Notable examples include the theory of matroid and polymatroid valuations \cite{DerksenFink,JocSan}, lattice point enumeration \cite{Brion}, and subword complexes of Coxeter systems \cite{Stump}.  Our goal is to understand matroids by viewing them as a special class of unbounded matroids.

Having given this higher-level overview, we now describe the structure of the article.

Section~\ref{sec:background} reviews the basic objects we will need: posets and lattices, polyhedra and generalized permutahedra, and submodular systems.

In Section~\ref{sec:UM}, we introduce unbounded matroids both geometrically (as 0/1 generalized permutahedra) and combinatorially (as a special case of submodular systems).  We include examples to illustrate how these new objects are both like and unlike matroids; a central example is a U-matroid we call the \defterm{stalactite} (Example~\ref{ex:stalactite}).

Section~\ref{sec:sheared} studies relationships between U-matroid polyhedra $\ppp,\ppp'$.  By the theory of submodular systems restricting the underlying lattice of a submodular system enlarges the base polyhedron, but can remove vertices.  Accordingly, we say that $\ppp'$ is a \defterm{sheared polyhedron} of $\ppp$ if $\ppp'\subseteq\ppp$, but the vertices of~$\ppp'$ contain those of~$\ppp$.
Indeed, the sheared polyhedra of~$\ppp$ correspond exactly to extensions of its rank function  (Theorem~\ref{thm:sheared-correspondence}).  When $U=(E,\D,\rho)$ is a U-matroid and $\D'$ is a distributive lattice obtained from $\D$ by adjoining a single atom, we construct an explicit extension of $\rho$ to $\D'$, which we call the \defterm{generous atom extension} (Definition~\ref{defn:generous-atom} and Theorem~\ref{thm:Dextend}).  By iterating this construction, we can extend U-matroid rank functions to any distributive lattice $\D''\supseteq\D$, and the U-matroid rank function thus obtained dominates every other extension of~$\rho$ to~$\D''$ (Theorem~\ref{generous-dominates-full}).  In particular, there is a canonical extension of $U$ to a matroid on the same ground set, corresponding to the largest matroid polytope contained in the base polyhedron of~$U$.

In Section~\ref{sec:basis}, we study the \defterm{pseudo-independence complex} of a U-matroid $U$, defined as the simplicial complex $\Delta(U)$ generated by its bases.  We first show (Theorem~\ref{thm:shelling}) that it is shellable, using the following geometric argument: first reduce the base polyhedron $\ppp$ to a polytope~$\qqq$ by slicing with an appropriately chosen hyperplane, then show that any line shelling of~$\qqq$ that starts with the hyperplane facet gives rise to a shelling of $\ppp$.  A consequence of the proof is that any linear functional lying in the braid cone of a bounded direction of $\ppp$ (that is, indexed by a maximal chain in the characteristic lattice) gives rise to a shelling.  This observation enables us to generalize Bj\"orner's and Gale's characterizations of matroid complexes to U-matroids (Theorem~\ref{thm:basis-theorem}).  Essentially, one need only replace ``all orderings of the ground set'' in Bj\"orner's and Gale's theorems with ``all orderings arising from maximal chains in the characteristic lattice''.

In Section~\ref{sec:dual}, we describe duality for U-matroids, which parallels matroid theory very closely.  The definitions of duality in terms of the base polyhedron, basis systems, and rank functions go through with essentially no change, and duality behaves as expected with respect to restriction of the characteristic lattice.  (In contrast, geometric and combinatorial duality are not mutually compatible for polymatroids.)

Section~\ref{sec:PMatroid} studies the application of unbounded matroids to subspace arrangements.  We first review the poset matroids introduced by Barnabei et al.~\cite{Barnabei}, and show that they are precisely those U-matroids for which every basis belongs to the characteristic lattice.  We then review the theory of subspace arrangements and describe how their (polymatroid) rank functions are equivalent both to the poset matroids of~\cite{Barnabei} and to the multisymmetric matroids of Crowley et al.~\cite{Crowley}.  The main result of this section (Theorem~\ref{generous-geometry}) states that generous atom extension corresponds to the following operation on an arrangement: replace one of its spaces~$X$ by a pair $(X',H)$, chosen generically with $\codim X'=\codim X-1$, $\codim H=1$, and $X=X'\cap H$.

We close by listing several problems for future study in Section~\ref{sec:further}.

\textbf{Acknowledgements.}  JLM thanks the Departamento de Matem\'aticas at Pontificia Universidad Cat\'olica de Chile for their hospitality during his sabbatical in July--December 2023, during which this work was completed.

\section{Background: Lattices and submodular systems} \label{sec:background}

The symbol $[n]$ denotes the set $\{1,\dots,n\}$, and $\Sym_n$ denotes the symmetric group of permutations of~$[n]$.  For a finite set $E$, the symbol $\binom{E}{r}$ means the set of $r$-element subsets of $E$.  

\subsection{Posets and lattices} \label{subsec:lattices}
We assume familiarity with the basic theory of posets and lattices, particularly distributive and semimodular lattices; see, e.g., \cite[Chapter~II]{Aigner} or~\cite[\S\S3.3,3.4]{EC1}.

Throughout, let $E$ be a finite set of cardinality $n$.
The symbol $\Ord(E)$ denotes the set of bijections $\sigma:[n]\to E$, which we regard as total orders of $E$ via $\sigma(1)<\cdots<\sigma(n)$.
For a poset $\P$ on ground set $E$, let $\Le(\P)\subseteq \Ord(E)$ denote the set of linear extensions of $P$.  The dual poset $\P^*$ is defined by reversing all order relations in $\P$.

Each total ordering $\sigma\in\Ord(E)$ gives rise to two orders on the set $\binom{E}{r}$, as follows.
Let $A,B\in\binom{E}{r}$, with $A=\{a_1<\cdots<a_r\}$ and $B=\{b_1<\cdots<b_r\}$.
The \defterm{Gale order} (or \defterm{componentwise order}) is defined by $A\legale B$ if $a_i\leq b_i$ for all $i$; in general this is not a total order.
The \defterm{lexicographic order} is defined by $A\llex B$ if, for some $k$, we have $a_1=b_1$, $a_2=b_2$, \dots, $a_{k-1}=b_{k-1}$, $a_k<b_k$.
Lexicographic order is a total order and is a linear extension of Gale order, i.e., $A\llex B$ whenever $A\lgale B$.  If we wish to specify the underlying ordering $\sigma$, we may write $A\llex^\sigma B$, $A\legale^\sigma B$, etc.

All distributive lattices we will consider are subsets $\D\subseteq 2^E$ that contain $\0$ and $E$ and are closed under union and intersection, which give the join and meet operations respectively.  In addition, we always assume that $\D$ is \defterm{accessible}\footnote{Fujishige uses ``simple''.}: for every nonempty $A\in\D$, there exists $x\in A$ such that $A-x\in\D$.  (It is equivalent to require that every covering relation in $\D$ be of this form, or that some (or every) maximal chain in $\D$ has length $|E|$.)

The poset $\Irr(\D)$ of join-irreducibles of an accessible distributive lattice $\D\subseteq2^E$ induces a poset on $E$, by associating each element of $E$ with the unique smallest join-irreducible element of $\D$ that contains it.  For example, if $E=[3]$ and $\D$ is the distributive lattice $\{\0,1,2,12,23,123\}$ with join-irreducibles $\{1,2,23\}$, then we take $\Irr(\D)$ to be the poset on $E$ with the single relation $2<3$.

For every accessible lattice $\D\subseteq2^E$ and every $A\subseteq E$, we define
\[\supn_\D(A)=\bigcap\{B\in\D\st B\supseteq A\} \quad\text{and}\quad
\infn_\D(A)=\bigcup\{B\in\D\st B\subseteq A\}.\]
These are respectively the unique largest and unique smallest elements of $\D$ containing $A$.  Observe that
\begin{equation} \label{sup-facts}
\supn_\D(A\cup B)=\supn_\D(A)\cup\supn_\D(B) \quad\text{and}\quad
\supn_\D(A\cap B)\subseteq\supn_\D(A)\cap\supn_\D(B)
\end{equation}
although equality need not hold in the second containment: for example, take $\D=\{\0,a,ab\}$, $A=\{a\}$, $B=\{b\}$.

\subsection{Polyhedra and generalized permutahedra} \label{sec:polyhedra}

We assume familiarity with basic facts about polyhedra; see, e.g., \cite{Grunbaum,Ziegler}.  The symbol $V(\ppp)$ denotes the set of vertices of a polyhedron $\ppp$.  The (polyhedral) \defterm{cone} $\ccc=\Rr_{\geq0}\langle S\rangle$ generated by a finite set $S\subseteq\Rr^n$ is by definition the set of all nonnegative linear combinations of vectors in $S$.
The \defterm{dual} of $\ccc$ is the cone
$\ccc^*=\{\ell\in(\Rr^n)^*\st \ell(\xx)\geq 0\ \ \forall\xx\in \ccc\}.$
For each face $F$ of a polyhedron $\ppp\subseteq\Rr^n$, the (closed) \defterm{normal cone} is\footnote{It is more usual to define $N_\ppp(F)$ as the set of linear functionals maximized on $F$, rather than minimized.  We adopt this convention for compatibility with \cite{Fujishige}.}
\[N_\ppp(F) = \{\ell\in(\Rr^n)^* \st \ell(\xx)\leq\ell(\yy) \ \ \forall \xx\in F,\ \yy\in \ppp\}.\]
The normal cones are the faces of a polyhedral complex, the \defterm{normal fan} $\N(\ppp)$.  We write $|\N(\ppp)|$ for the union of the normal cones, that is, the cone of all linear functionals that are bounded from below on~$\ppp$.  In particular, $|\N(\ppp)|=(\Rr^n)^*$ if and only if $\ppp$ is bounded.  The \defterm{recession cone} (or \defterm{characteristic cone}) of $\ppp$ is
\[\rec(\ppp)=|\N(\ppp)|^*=\{\yy\in \Rr^n\st \xx+\yy\in \ppp \ \ \forall \xx\in \ppp\}\]
\cite[\S2.1]{CrawMaclagan}.
Informally, $\rec(\ppp)$ consists of ``all the directions in which $\ppp$ is unbounded.''  If $\ppp$ is bounded then $\rec(\ppp)=\{0\}$; otherwise, $\rec(\ppp)$ is unbounded.  Every polyhedron $\ppp$ can be expressed as a Minkowski sum $\ppp=\ppp'+\rec(\ppp)$, where $\ppp'$ is a polytope; this decomposition is unique if and only if $\ppp$ itself is a polytope.

Recall that the \defterm{braid arrangement} consists of the $\binom{n}{2}$ hyperplanes in $\Rr^n$ defined by the equations $x_i=x_j$ for $1\leq i<j\leq n$.  The faces of the braid arrangement are cones indexed by total preorders on $[n]$ (i.e., complete relations that are reflexive and transitive).
Specifically, a preorder $\preceq$ corresponds to the cone $\ccc_\preceq$ defined by
\[\ccc_\preceq = \{\xx=(x_1,\dots,x_n)\in\Rr^n \st x_i\leq x_j \iff i\preceq j\}.\]
In particular, the maximal cones correspond to permutations $\sigma\in\Sym_n$, where a permutation $\sigma\in\Sym_n$ is associated to the preorder $\preceq_\sigma$ such that $i\preceq_\sigma j\iff\sigma(i)\leq\sigma(j)$.

A polytope $\ppp$ is a \defterm{generalized permutahedron} \cite{Postnikov,PRW} if its normal fan is a coarsening of the braid fan.  That is, if we identify $(\Rr^n)^*$ with $\Rr^n$ by the standard dot product, then for every linear functional $\ell=(\ell_1,\dots,\ell_n)\in(\Rr^n)^*$, the face of $\ppp$ on which $\ell$ is minimized depends only on the relative order of the scalars $\ell_i$, i.e., on the cone of the braid arrangement in $(\Rr^n)^*$ containing $\ell$.
Equivalently, $\ppp$ is a generalized permutahedron if and only if each of its edges is parallel to the difference of two standard basis vectors.

An \defterm{extended generalized permutahedron} is a polyhedron~$\ppp$ whose normal fan is a coarsening of a subfan of the braid fan: that is,  for any linear functional $\ell\in(\Rr^n)^*$ \textit{that is bounded from below on $\ppp$}, the face of~$\ppp$ minimized by $\ell$ depends only on the braid cone containing it.  Equivalently, every 1-dimensional face (both bounded edges and unbounded rays) is parallel to the difference of two standard basis vectors.

\subsection{Submodular systems} \label{sec:submodular}

\begin{defn}[\textbf{Submodular systems}] \label{def:submodsyst} \cite[pp.33--34]{Fujishige}
A \defterm{submodular system} is a triple
$S=(E,\D,\rho)$, where $E$ is a finite ground set (typically $[n]$), $\D\subseteq2^E$ is an accessible lattice, and $\rho:\D\to\Rr$ is a submodular \textit{rank function} satisfying the following additional properties for all $A,B\in\D$:
\begin{align}
& \rho(\0)=0  && \textup{(calibration);}\label{calibration}\\
& A\subseteq B\implies\rho(A)\leq\rho(B) && \textup{(monotonicity);} \label{monotonicity}\\
& \rho(A)+\rho(B)\geq \rho(A\cup B)+\rho(A\cap B) && \textup{(submodularity).} \label{submodularity}
\end{align}
\end{defn}

Monotonicity is absent from the definition in \cite{Fujishige}, but it is natural in a combinatorial context.

A \defterm{polymatroid} \cite[Defn.~3, p.338]{Welsh}; \cite[\S2.2, pp.25--26]{Fujishige} is a submodular system for which $\D=2^E$.  In the further special case that $\rho(A)\in\Nn$ and $\rho(A)\leq|A|$ for all $A$, we recover the definition of the rank function of a matroid.

For a vector $\xx=(x_1,\dots,x_n)\in\Rr^n$ and a set $A\subseteq[n]$, define
$\xx(A)=\sum_{a\in A} x_i.$
The \defterm{submodular polyhedron} and the \defterm{base polyhedron} of a submodular system $S=([n],\D,\rho)$ are respectively
\begin{align}
\AAA(S) &= \left\{\xx\in\Rr^n \st \xx(A)\leq\rho(A) \ \forall A\in\D\right\}, \label{define-submod-poly} \\
\BBB(S) &= \left\{\xx\in \AAA(S) \st \xx(E)=\rho(E)\right\}.\label{define-base-poly}
\end{align}
We will primarily be interested in the base polyhedron $\BBB(S)$, which is both a face of $\AAA(S)$ and an extended generalized permutahedron.   The correspondence between submodular systems and extended generalized permutahedra is in fact a bijection \cite[Thm.~3.1.6]{AA}, and much of the geometry of $\BBB(S)$ can be described explicitly in terms of $S$ \cite[\S2.3]{Fujishige}.

\begin{remark}
When $S$ is a matroid, the base polyhedron $\BBB(S)$ is a polytope, namely the convex hull of the characteristic vectors of its bases (the usual definition of the base polytope of a matroid).  On the other hand, the submodular polyhedron $\AAA(S)$ is still unbounded (its recession cone contains the negative orthant).  In particular, it does not coincide with the independence polytope $\III(S)$, defined as the convex hull of the characteristic vectors of independent sets in $M$.  (But in fact $\III(S)=\AAA(S)\cap[0,1]^E$; see Theorem~\ref{generous-polytopes} below.)
\end{remark}

\begin{defn}
\label{def:bounded-direction}
Let $E=[n]$, let $\sigma=(\sigma(1),\dots,\sigma(n))\in\Sym_n$, let $S=(E,\D,\rho)$ be a submodular system, and let $\P=\Irr(\D)$.  We say that $\sigma$ is \defterm{bounded with respect to $S$} if the following equivalent conditions hold:
\begin{enumerate}
\item\label{bounded:min} The linear functional $\xx\mapsto\sum_i\sigma^{-1}(i)\xx_i$ achieves a minimum on $\BBB(S)$.  (Equivalently,
the braid cone $\ccc_{\sigma^{-1}}$ is contained in the normal fan of $\BBB(S)$.)
\item\label{bounded:linex} $\sigma\in\Le(\P)$, i.e., if $i<_\P j$ then $\sigma(i)<\sigma(j)$;
\item\label{bounded:init} For every $k\in[n-1]$, the initial segment $\{\sigma(1),\dots,\sigma(k)\}$ belongs to $\D$.
\end{enumerate}
\end{defn}

Here the equivalence of (\ref{bounded:min}) and (\ref{bounded:linex}) is \cite[Thm.~3.13]{Fujishige}, and the equivalence of (\ref{bounded:linex}) and (\ref{bounded:init}) is a consequence of Birkhoff's theorem (where the ground set of $\P$ is identified with $E$, as in Section~\ref{subsec:lattices}).
In particular, $\BBB(S)$ is bounded if and only if $S$ is a polymatroid (i.e., $\D=2^E$).  In general, the smaller the lattice $\D$, the larger the recession cone.

We now review some known facts about the geometry of $\BBB(S)$.

\begin{thm}[\textbf{Vertices of base polyhedra of submodular systems}] \label{thm:submod-vertex}
\cite[Thm.~3.22, p.67]{Fujishige}
Let $S=(E,\D,\rho)$ be a submodular system.  For each $\sigma\in\Ord(E)$ that is bounded with respect to $S$ (i.e., $\sigma\in\Le(\Irr(\D))$), define a point $\xx=\xx^\rho_\sigma=(x_1,\dots,x_n)\in\Rr^n$ by
\begin{equation} \label{x-from-sigma}
x_{\sigma(i)} = \rho(\{\sigma(1),\dots,\sigma(i)\})-
\rho(\{\sigma(1),\dots,\sigma(i-1)\}), \quad 1\leq i\leq n.
\end{equation}
Then the vertices of the base polyhedron $\BBB(S)$ are exactly the points $\xx^\rho_\sigma$ defined by~\eqref{x-from-sigma}.
\end{thm}

Observe that monotonicity of $S$ implies that the vertices of the base polyhedron lie in the principal orthant in $\Rr^E$; in fact these conditions are equivalent \cite[Lemma~3.23]{Fujishige}.  Moreover, $\D$ is accessible if and only if the base polyhedron $\BBB(S)$ has at least one vertex \cite[Thm.~3.11]{Fujishige}.

The inequalities $\xx(A)\leq\rho(A)$ are in fact tight on the base polyhedron:
\begin{prop} \label{prop:inequalities-are-tight}
Let $S=(E,\D,\rho)$ be an accessible submodular system and $A\in\D$.  Then $\BBB(S)$ has a vertex for which $\xx(A)=\rho(A)$.  In particular, for each $A\in\D$, the inequality $\xx(A)\leq\rho(A)$ defines a face of $\BBB(S)$.
\end{prop}
\begin{proof}
Let $\sigma\in\Ord(E)$ be a linear order that is bounded with respect to $\D$ and for which $A$ is an initial segment.  Then the point $\xx=\xx^\rho_\sigma$ defined by~\eqref{x-from-sigma} is a vertex of $\BBB(S)$, and satisfies $\xx(A)=\rho(A)$.
\end{proof}

\begin{thm}[\textbf{Recession cones of base polyhedra of submodular systems}]\label{thm:extreme-ray}
\cite[Thm.~3.26, p.70]{Fujishige}
Let $S=(E,\D,\rho)$ be a submodular system, and let $\P$ be the poset on $E$ with relations $i\leq j$ whenever every element of $\D$ containing $j$ also contains $i$.  Equivalently, $\P\isom\Irr(\D)$, the poset of join-irreducible elements, in which $e\in E$ corresponds to the unique smallest element of~$\D$ containing~$e$.
Then the recession cone of $\BBB(S)$ depends only on $\D$ (not on $\rho$) and it is
\[
\rec(\D) = \Rr_{\geq0}\langle\ee_j-\ee_i \st j>_\P i\rangle.
\]
\end{thm}
Observe that this theorem also completely describes the poset $\P$, and thus the distributive lattice~$\D$, in terms of the recession cone of the base polyhedron.

\section{Unbounded matroids} \label{sec:UM}

We can now define the central objects of our study, both geometrically and combinatorially.

\begin{defn}
A polyhedron $\ppp\subset\Rr^n$ is an \defterm{unbounded matroid polyhedron}, or for short a \defterm{U-polyhedron}, if the following conditions hold:
\begin{enumerate}
\item $\ppp$ is an extended generalized permutahedron (that is, every 1-dimensional face, whether bounded or unbounded, is parallel to a difference of two standard basis vectors); 
\item Every vertex of $\ppp$ is an 0,1-vector;
\item $\ppp$ has at least one vertex,
\end{enumerate}
\end{defn}
The word ``unbounded'' should be interpreted as ``possibly unbounded''.  In particular, a bounded U-polyhedron is just a matroid base polytope.

\begin{defn} \label{def:U-matroid}
A \defterm{U-matroid} is a submodular system $U=(E,\D,\rho)$ (see Defn.~\ref{def:submodsyst}) satisfying the following additional conditions for all $A,B\in\D$:
\begin{align}
& \rho(A)\in\Nn && \textup{(integrality);}\label{integrality}\\
& B=A\cup\{e\}\implies\rho(B)-\rho(A)\leq 1 && \textup{(unit increase).} \label{unit-increase}
\end{align}
The unit-increase property is equivalent to the condition that if $A\subseteq B$, then $\rho(B)-\rho(A) \leq |B|-|A|$.
\end{defn}

We call $\D$ the \defterm{characteristic lattice} of $U$, and $\P=\Irr(\D)$ is its \defterm{characteristic poset}.  By Birkhoff's theorem, each of $\D$ and $\P$ determines the other, so when defining a U-matroid we may specify whichever is more convenient.

This definition is very close to that of the rank function of a pregeometry in a partial order, as defined by Faigle~\cite[Lemma~5]{Faigle}; we will discuss the connection in more detail momentarily.

Observe that a U-matroid is a matroid if and only if $\D=2^E$.  That is, the intersection of the class of polymatroids and the set of U-matroids is precisely the class of matroids.

\begin{thm} \label{thm:ext-mat-GP}
The correspondence given by Theorem~\ref{thm:submod-vertex} restricts to a bijection between U-matroids and extended 0/1 generalized permutahedra.
\end{thm}
\begin{proof}
Let $U=(E,\D,\rho)$ be a U-matroid. Theorem~\ref{thm:submod-vertex} gives that $\BBB(U)$ is a polyhedron with vertices $\{\xx^\rho_\sigma\ |\ \sigma\in\Ord(E)\}$, where $\xx^\rho_\sigma$ is defined by \eqref{x-from-sigma}. It follows by the unit increase property that any such vertex is 0/1.

Now let $(E,\D,\rho)$ be a submodular system that is not a U-matroid. Then there exist $A\in\D$ and $e\in E$ with $A\cup\{e\}\in\D$ and $\rho(A\cup\{e\})\not\in\{\rho(A),\rho(A)+1\}$.  Take any maximal chain of $\D$ passing through $A$ and $A\cup\{e\}$.  By Theorem~\ref{thm:submod-vertex} this chain gives rise to a vertex of the base polyhedron whose $e$th coordinate is not 0/1.
\end{proof}

Recall from \S\ref{sec:polyhedra} that the normal fan of a U-polyhedron $\ppp$ is a coarsening of a maximal-dimensional convex subfan of the braid fan.  The support of that normal fan corresponds to a poset $\P$; specifically, the support is the union of the chambers of the braid arrangement that correspond to the linear extensions of~$\P$.  Thus $\P$ is in fact the characteristic poset of the U-matroid corresponding to $\ppp$.  The U-matroid is a matroid, and its base polyhedron is bounded, if and only if the lattice $\D$ is Boolean (equivalently, $\P$ is an antichain).

\begin{example}
Two U-polyhedra $\Blue{\ppp_1},\Red{\ppp_2}\subset\Rr^3$ are shown in Figure~\ref{fig:twoU}.  They have the same recession cone, namely the ray $\Rr_{\geq0}\langle\ee_3-\ee_2\rangle$, hence the same characteristic poset (with ground set $E=[3]$ and one relation $2<3$) and the same characteristic lattice (namely $\{\0,1,2,12,23,123\}$).  On the other hand, their rank functions $\Blue{\rho_1},\Red{\rho_2}$ differ.

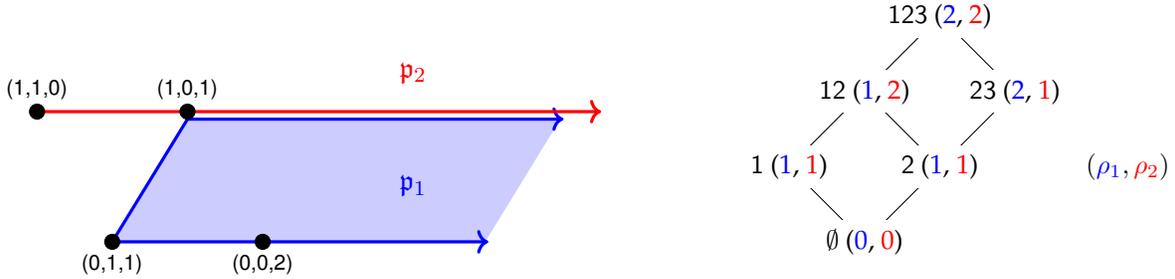
\begin{figure}[ht]
\begin{center}
\begin{tikzpicture}
\newcommand{\scl}{2}
\newcommand{\epsil}{.1}
\draw[blue, <->, very thick, fill=blue!20] (\scl*2.5,0)--(\scl*0,0)--(\scl*.5,\scl*.866-\epsil)--(\scl*3,\scl*.866-\epsil);
\draw[red, ->, very thick] (-\scl*0.5,\scl*.866)--(\scl*3.25,\scl*.866);
\draw[fill=black] (0,0) circle (\scl*.05);
\foreach \a in {0,60,120} \draw[fill=black] (\a:\scl) circle(\scl*.05);
\node at (\scl*0,-.3) {\sf\footnotesize(0,1,1)};
\node at (\scl*1,-.3) {\sf\footnotesize(0,0,2)};
\node at (\scl*.5,\scl*.866+.3) {\sf\footnotesize(1,0,1)};
\node at (-\scl*.5,\scl*.866+.3) {\sf\footnotesize(1,1,0)};
\node[blue] at (\scl*2,.75) {\large $\ppp_1$};
\node[red] at (\scl*2,2.25) {\large $\ppp_2$};

\begin{scope}[shift={(10,0)}]
 
\foreach \a in {0,1} \draw (-\a,\a)--(-\a+2,\a+2);
\foreach \a in {0,1,2} \draw (\a,\a)--(\a-1,\a+1);
\foreach \x/\y/\l/\r/\s in {
				1/3/123/2/2,
		0/2/12/1/2,		2/2/23/2/1,
-1/1/1/1/1,			1/1/2/1/1,
		0/0/\0/0/0}
\node[fill=white] at (\x,\y) {$\mathsf{\l}$ (\Blue{\r}, \Red{\s})};
\node at (3.5,1) {$(\Blue{\rho_1},\Red{\rho_2})$};
\end{scope}
\end{tikzpicture}
\end{center}
\caption{Two U-polyhedra with the same characteristic poset but different rank functions.\label{fig:twoU}}
\end{figure}
\end{example}

\begin{defn} \label{def:basis}
A \defterm{basis} of a U-matroid $U=(E,\D,\rho)$ is the support of a vertex of its base polyhedron.
\end{defn}
The set of all bases will be denoted $\B(U)$ or $\B(\rho)$.  Note that the definition of a U-matroid polyhedron implies that $\B(U)\neq\0$.
By Theorem~\ref{thm:submod-vertex}, every basis arises from a maximal chain $A=A_0\coveredby\cdots\coveredby A_n$ in $\D$, or equivalently from the corresponding linear extension $\sigma\in\P$ with $A_k=\{\sigma(1),\dots,\sigma(k)\}$.  The basis records the ``rank jumps'' in $A$:
\begin{equation} \label{basis-from-linext}
B_\rho(A)=B_\rho(\sigma)=\left\{\sigma(i)\st \rho(A_i)=\rho(A_{i-1})+1\}\right\}.
\end{equation}
That is, $B_\rho(\sigma)=\{\sigma(i_1),\dots,\sigma(i_{\rho(E)})\}$, where $i_1<\cdots<i_r$ and, for each $k\in[r]$, the set $\{\sigma(i_1),\dots,\sigma(i_k)\}$ is the $\sigma$-lexicographically smallest set of cardinality $k$ and rank $k$.

Theorem~\ref{thm:submod-vertex} describes how to calculate the basis system of a U-matroid from its rank function.  The next result describes how to recover the rank function from the vertex set and the characteristic lattice.

\begin{prop} \label{prop:rank-from-vertices}
The rank function of a U-matroid $U=(E,\D,\rho)$ is determined by the set $\B(U)$ and the lattice $\D$ by the formula
\begin{equation} \label{rank-vertex}
\rho(A) = \max\{|A\cap B|\st B\in\B(U)\}.
\end{equation}
\end{prop}

\begin{proof}
First, for every maximal chain $\0=A_0\coveredby\cdots\coveredby A_n=[n]$ in $\D$, let $B=B_\rho(A)$.  Then $|A_k\cap B|=\rho(A_k)$ for every $k$, implying the $\leq$ direction of~\eqref{rank-vertex}.  For the converse, let $C\in\D$.  Consider the chain
\begin{equation} \label{Cchain}
\0=A_0\cap C\subset A_1\cap C\cdots\subset A_n\cap C=C
\end{equation}
in $\D$ (which includes repeats).  If the unique element of $A_i\sm A_{i-1}$ belongs to $C$ and $\rho(A_i)=\rho(A_{i-1})+1$, then by submodularity
\[\rho(A_i\cap C)\geq \rho(A_i)+\rho(A_{i-1}\cap C)-\rho(A_{i-1})\geq \rho(A_{i-1}\cap C)+1\]
so equality must hold.  It follows that the number of rank jumps in the chain~\eqref{Cchain} is at least $|B\cap C|$, and that number of rank jumps is just $\rho(A)$, establishing the reverse inequality.
\end{proof}

\begin{cor} \label{easy-basis}
If $A\in\D$ and $\rho(A)=|A|=\rho(E)$, then $A$ is a basis.
\end{cor}

On the other hand, the connection between bases and the rank function is more complicated for U-matroids than for matroids, and some standard statements of matroid theory turn out to fail in the setting of U-matroids.  For instance, not every minimal element of full rank in the characteristic lattice~$\D$ is a basis.  In fact, not every basis must be an element of $\D$.

\begin{example} \label{ex:not-comb-scheme}
Consider the U-matroid with characteristic lattice $\D$ and rank function \Red{$\rho$} shown below.
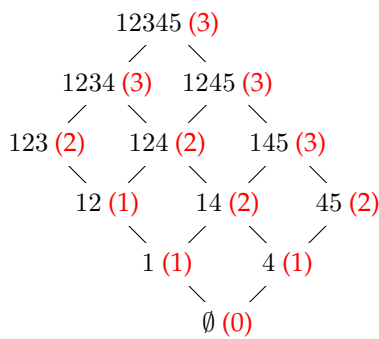
\begin{center}
\begin{tikzpicture}[scale=0.8]
\foreach \a in {0,1,2} \draw (\a,\a)--(\a-3,\a+3);
\foreach \a in {0,1,2,3} \draw (-\a,\a)--(2-\a,2+\a);
\foreach \x/\y/\l/\r in {
0/0/\0/0,  1/1/4/1,    2/2/45/2,
-1/1/1/1,  0/2/14/2,   1/3/145/3,
-2/2/12/1, -1/3/124/2, 0/4/1245/3,
-3/3/123/2,-2/4/1234/3,-1/5/12345/3}
\node[fill=white] at (\x,\y) {$\l$ \color{red}(\r)};
\end{tikzpicture}
\end{center}
The minimal elements of full rank are $145$ and $1234$.  The former is a basis (by Corollary~\ref{easy-basis}), but the latter cannot be because it has the wrong cardinality.  In fact the other basis is $134$ (arising from, e.g., the chain $\0<4<14<124<1234<12345$), which does not belong to~$\D$.
\end{example}

Basis systems of U-matroids are pure set families (i.e., all members have the same cardinality), but they are strictly more general than matroid basis systems.  The following U-matroid, which was first described in \cite[p.~31]{CMS}, is a core example: its vertex set is not the vertex set of a matroid base polytope.

\begin{example} \label{ex:stalactite}
The \defterm{stalactite}, shown in Figure~\ref{fig:stalactite}, is the U-polyhedron
\[\sss=\{\xx\in\Rr^4\st x_1+x_2+x_3+x_4=2,\ x_2,x_3,x_4\geq 0,\ x_1,x_2,x_3\leq 1\}.\]
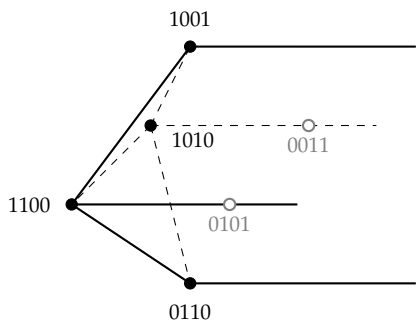
\begin{figure}[htb]
\begin{center}
\begin{tikzpicture}[scale=0.75]
\newcommand{\scl}{0.7}
\newcommand{\xiooi}{0*\scl}	\newcommand{\yiooi}{6*\scl}
\newcommand{\xooii}{3*\scl}	\newcommand{\yooii}{4*\scl}
\newcommand{\xioio}{-1*\scl}	\newcommand{\yioio}{4*\scl}
\newcommand{\xoioi}{1*\scl}	\newcommand{\yoioi}{2*\scl}
\newcommand{\xiioo}{-3*\scl}	\newcommand{\yiioo}{2*\scl}
\newcommand{\xoiio}{0*\scl}	\newcommand{\yoiio}{0*\scl}
\coordinate (oiio) at (\xoiio,\yoiio);
\coordinate (ioio) at (\xioio,\yioio);
\coordinate (ooii) at (\xooii,\yooii);
\coordinate (iioo) at (\xiioo,\yiioo);
\coordinate (oioi) at (\xoioi,\yoioi);
\coordinate (iooi) at (\xiooi,\yiooi);
\foreach \coo in {(oiio),(iooi),(iioo)} \draw[dashed] (ioio)--\coo;
\draw[thick] (oiio)--(iioo)--(iooi);
\draw[dashed] (ioio)--(4+\xioio,\yioio);
\draw[thick] (iooi)--(4+\xiooi,\yiooi);
\draw[thick] (oiio)--(4+\xoiio,\yoiio);
\draw[thick] (iioo)--(4+\xiioo,\yiioo);
\draw[thick,gray,fill=white] (\xooii,\yooii) circle(.1); \node[gray] at (\xooii,\yooii-.35) {\footnotesize0011};
\draw[thick,gray,fill=white] (\xoioi,\yoioi) circle(.1); \node[gray] at (\xoioi,\yoioi-.35) {\footnotesize0101};
\draw[fill=black] (\xoiio,\yoiio) circle(.1); \node at (\xoiio,\yoiio-.5) {\footnotesize0110};
\draw[fill=black] (\xioio,\yioio) circle(.1); \node at (\xioio+.75,\yioio-.25) {\footnotesize1010};
\draw[fill=black] (\xiioo,\yiioo) circle(.1); \node at (\xiioo-.75,\yiioo) {\footnotesize1100};
\draw[fill=black] (\xiooi,\yiooi) circle(.1); \node at (\xiooi,\yiooi+.5) {\footnotesize1001};
\end{tikzpicture}
\caption{The stalactite.\label{fig:stalactite}}
\end{center}
\end{figure}

We now describe the corresponding U-matroid.
The recession cone is the ray spanned by $(-1,0,0,1)$, so by Theorem~\ref{thm:extreme-ray} the characteristic poset $\P$ on $E=\{1,2,3,4\}$ has only one relation, namely $1<4$.  Thus the characteristic lattice $\D$ is $J(\P)=\{A\subseteq[4]\st 4\in A\implies 1\in A\}$.
By Proposition~\ref{prop:rank-from-vertices}, the rank function $\rho$ is given by $\rho(A) = \min(|A|,2)$ for all $A\in D$.

The supports of the vertices of the stalactite do not form a matroid basis system.  It may be easiest to see this geometrically: the convex hull of the vertices is a simplex, and the edge between $(1,0,0,1)$ and $(0,1,1,0)$ is not parallel to the difference of two standard basis vectors.  (On the other hand, $\sss$ is the Minkowski sum of its recession cone with the base polytope of the uniform matroid $U_2(4)$, whose vertices include the four black and two white points shown in Figure~\ref{fig:stalactite}.  As we will see, every U-matroid is the Minkowski sum of its recession cone with a matroid base polytope.) 
\end{example}

We now describe the connection to the work of Faigle~\cite{Faigle}.  Let $\P$ be a partial order on $E$.
Faigle~\cite[\S2]{Faigle} defined a \defterm{pregeometric closure operator} on $\D=J(\P)$ to be a map $\cl:2^E\to\D$
satisfying the following conditions for all $A,B\subseteq E$ and $p,q\in E$:
\begin{enumerate}[label=(C\arabic*)]
\item\label{Faigle:D1} $A\subseteq\clo{A}$.
\item\label{Faigle:D2} if $A\subseteq\clo{B}$ then $\clo{A}\subseteq\clo{B}$; and
\item\label{Faigle:D3} Suppose that $r\in\clo{A}$ for all $r<p$; $q\notin\clo{A}$; and $q\in\clo{A\cup p}$.  Then $p\in\clo{A\cup q}$.
\end{enumerate}

A set of the form $\clo{A}$ is called a \defterm{flat}.
The set $\L$ of all flats, regarded as a poset under containment, is in fact a semimodular lattice whose meet is intersection~\cite[Prop.~3]{Faigle}.  In particular, $\L$ is ranked, and its rank function satisfies the axioms of Definition~\ref{def:U-matroid} \cite[Lem.~5]{Faigle}.  Conversely, by \cite[Thm.~6, Prop.~10]{Faigle}, every U-matroid $U=(E,\D,\rho)$ gives rise to the pregeometric closure operator $\cl=\cl_U:2^E\to\D$ defined by
\begin{equation}\label{ranktoclosure}
\clo{A}=\{e\in E\ |\ \rho(\supn_\D(A\cup e))=\rho(\supn_\D(A))\},
\end{equation}
with associated lattice of flats
\begin{subequations}
\begin{equation}
\L(U)=\L(\rho)= \{A\in\D\st \rho(B)>\rho(A) \ \ \forall B\in \D\st B\supsetneq A\} \label{flats-from-rank-alt}
\end{equation}
and the correspondences between rank functions and closure operators (resp., lattices of flats) are bijections.
For later use, we observe that~\eqref{flats-from-rank-alt} may be rewritten as
\begin{align}
\L(U)=\L(\rho)
&= \{A\in\D\st \rho(A\cup e)>\rho(A) \ \ \forall e\in E\sm A \textup{ with } A\cup e\in\D\} \label{flats-from-rank-no-sup}\\
&= \{A\in\D\st \rho(\supn_\D(A\cup e))>\rho(A) \ \ \forall e\in E\sm A\}. \label{flats-from-rank}
\end{align}
\end{subequations}

\begin{remark}
Faigle~\cite[pp.~39,43]{Faigle} defined two kinds of bases: a minimal element of $\D$ whose closure is $E$ (a ``basis''), and a tuple $(x_1,\dots,x_r)$ of elements of~$E$ such that $r=\rho(E)$ and $\rho(\{x_1,\dots,x_k\})=k$ for all $k$ (a ``B-basis'').  By the last example, neither definition coincides with ours.  The set 1234 is a basis in Faigle's sense but not in ours, and the B-bases  include tuples such as  $(2,3,4)$ that do not correspond to vertices.
\end{remark}

\section{Sheared polyhedra and generous extensions} \label{sec:sheared}

With our combinatorics--geometry dictionary in hand, we wish to study relationships among U-matroids on the same ground set.  A natural question is when a given U-matroid $U=(E,\D,\rho)$ can be extended to a matroid on $E$, or more generally to a U-matroid $(E,\D',\rho')$ with $\D\subseteq\D'\subseteq 2^E$.  Geometrically, such a lattice extension corresponds to constructing a special subpolyhedron of $\BBB(U)$ that we call a \defterm{sheared polyhedron} (or a \defterm{sheared polytope} in the case $\D'=2^E$).  In fact, we will construct a canonical extension of $\rho$ to $\D'$, the \defterm{generous extension}, which dominates every other lattice extension of $\rho$.  When $\D'=2^E$, this construction gives the unique largest matroid base polytope contained in any U-matroid base polyhedron.

\subsection{Sheared polyhedra and lattice extensions} \label{sec:sheared-basics}

\begin{defn} \label{defn:sheared}
Let $\ppp$ and $\ppp'$ be U-matroid polyhedra.  We say that $\ppp'$ is a \defterm{sheared polyhedron} of $\ppp$ if $\ppp'\subseteq \ppp$ and $V(\ppp')\supseteq V(\ppp)$.  A \defterm{sheared polytope} is a bounded sheared polyhedron.
\end{defn}

\newcommand{\ssstop}{\sss_\textup{top}}
\begin{example}\label{ex:halfstalact}
Let $\ssstop$ be the subpolyhedron of the stalactite~$\sss$ (see Example~\ref{ex:stalactite}) satisfying $x_1+x_4\geq1$.  The square pyramid $\ssstop'$ shown in violet below is a sheared polytope of $\ssstop$.
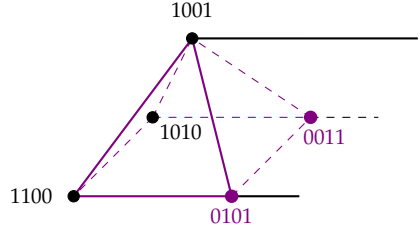
\begin{center}
\begin{tikzpicture}[scale=0.75]
\newcommand{\scl}{0.7}
\newcommand{\xiooi}{0*\scl}	\newcommand{\yiooi}{6*\scl}
\newcommand{\xooii}{3*\scl}	\newcommand{\yooii}{4*\scl}
\newcommand{\xioio}{-1*\scl}	\newcommand{\yioio}{4*\scl}
\newcommand{\xoioi}{1*\scl}	\newcommand{\yoioi}{2*\scl}
\newcommand{\xiioo}{-3*\scl}	\newcommand{\yiioo}{2*\scl}
\coordinate (ioio) at (\xioio,\yioio);
\coordinate (ooii) at (\xooii,\yooii);
\coordinate (iioo) at (\xiioo,\yiioo);
\coordinate (oioi) at (\xoioi,\yoioi);
\coordinate (iooi) at (\xiooi,\yiooi);
\foreach \coo in {(iooi),(iioo)} \draw[dashed,violet] (ioio)--\coo;
\draw[thick, violet] (iioo)--(iooi);
\draw[dashed, violet] (oioi)--(ooii)--(iooi);
\draw[thick, violet] (iooi)--(oioi);
\draw[dashed, violet] (ioio)--(ooii);
\draw[dashed] (ooii)--(4+\xioio,\yioio);
\draw[thick] (iooi)--(4+\xiooi,\yiooi);
\draw[thick, violet] (iioo)-- (oioi);
\draw[thick] (oioi)--(4+\xiioo,\yiioo);
\draw[thick,violet,fill=violet] (\xooii,\yooii) circle(.1); \node[violet] at (\xooii+.25,\yioio-.35) {\footnotesize0011};
\draw[thick,violet,fill=violet] (\xoioi,\yoioi) circle(.1); \node[violet] at (\xoioi,\yoioi-.35) {\footnotesize0101};
\draw[fill=black] (\xioio,\yioio) circle(.1); \node at (\xioio+.5,\yioio-.25) {\footnotesize1010};
\draw[fill=black] (\xiioo,\yiioo) circle(.1); \node at (\xiioo-.75,\yiioo) {\footnotesize1100};
\draw[fill=black] (\xiooi,\yiooi) circle(.1); \node at (\xiooi,\yiooi+.5) {\footnotesize1001};
\end{tikzpicture}
\end{center}
\end{example}

\begin{defn} \label{defn:lattice-restriction}
Let $U=(E,\D,\rho)$ be a U-matroid and $\D'\subseteq\D$ an accessible distributive sublattice.  The \defterm{lattice restriction} of $U$ to $\D'$ is the U-matroid $U'=U|_{\D'}=(E,\D',\rho|_{\D'})$.  In this case we say that $U$ is a \defterm{lattice extension} of $U'$ to $\D$.
\end{defn}

\begin{thm}\label{thm:sheared-correspondence}
Let $U=(E,\D,\rho)$ and $U'=(E,\D',\rho')$ be U-matroids, with $\ppp=\BBB(U)$ and $\ppp'=\BBB(U')$.  Then $U'$ is a lattice extension of $U$
if and only if $\ppp'$ is a sheared polyhedron of $\ppp$.
\end{thm}

\begin{proof}
($\Longrightarrow$) Suppose that $U'$ is a lattice extension of $U$.  The equations~\eqref{define-base-poly} defining $\ppp$ are a subset of those defining $\ppp'$, so $\ppp\supseteq \ppp'$.  Moreover, every direction that is bounded with respect to $\D$ is bounded with respect to $\D'$, so the description of vertices of base polyhedra of submodular systems (Theorem~\ref{thm:submod-vertex}) implies that $V(\ppp)\subseteq V(\ppp')$.
\medskip

($\Longleftarrow$) Suppose that $\ppp'$ is a sheared polyhedron of $\ppp$.  Since $\ppp'\subseteq \ppp$, every direction that is bounded with respect to $\ppp$ is bounded with respect to $\ppp'$.  Therefore $\D\subseteq\D'$.  It remains to show that $\rho'|_\D=\rho$.

Let $A\in\D$.  By Prop.~\ref{prop:inequalities-are-tight}, there exist vertices $\xx\in V(\ppp)$ and $\xx'\in V(\ppp')$ such that $\xx(A)=\rho(A)$ and $\xx'(A)=\rho'(A)$.  On the other hand, the definition of a sheared polyhedron implies that $V(\ppp)\subset \ppp'$ and $V(\ppp')\subset \ppp$.  Therefore, by~\eqref{define-base-poly}, we have
\[\rho(A)=\xx(A)\leq\rho'(A)=\xx'(A)\leq\rho(A)\]
and so equality holds throughout.  Thus $\rho'|_\D=\rho$.
\end{proof}

\begin{cor}\label{cor:recession-sum}
Let $U=(E,\D,\rho)$ and $U'=(E,\D',\rho')$ be U-matroids such that $U$ is a lattice restriction of $U'$. Then $\BBB(U')+\rec(\D)=\BBB(U)$.
(See Theorem~\ref{thm:extreme-ray} for a description of $\rec(\D)$.)
\end{cor}
\begin{proof}
    By Theorem~\ref{thm:sheared-correspondence}, $\BBB(U)$ is a sheared polyhedron of $\BBB(U')$.
    Then $\BBB(U')\subseteq \BBB(U)$ and $\rec(\D)$ is the recession cone of $\BBB(U)$, so it is immediate that $\BBB(U')+\rec(\D)\subseteq \BBB(U)$.
    On the other hand, every point in $\BBB(U)$ is a convex combination of the vertices of $V(\BBB(U))$ plus a vector in $\rec(\D)$. Since $V(\BBB(U'))\supseteq V(\BBB(U))$, every such element is also in $\BBB(U')+\rec(\D)$.
\end{proof}

We now characterize sheared polyhedra of a U-matroid polyhedron using Minkowski sums.

\begin{prop}\label{minkmat}
Let $\ppp,\ppp'$ be U-matroid polyhedra. Then $\ppp'$ is a sheared polyhedron of $\ppp$ if and only if $\ppp$ is the Minkowski sum $\ppp'+\rec(\ppp)$.
\end{prop}
See Example~\ref{ex:halfstalact}, where $\ssstop$ is the Minkowski sum of the sheared polytope $\ssstop'$ with the recession cone of $\ssstop$ (namely, the ray generated by $e_4-e_1$).
\begin{proof}
We first show that if $\ppp'$ is a sheared polyhedron of $\ppp$, then $\ppp'+\rec(\ppp)=\ppp$. Since $\ppp'\subseteq \ppp$ (by definition of sheared polyhedron), it is immediate that $\ppp'+\rec(\ppp)\subseteq \ppp$.  On the other hand, every point in $\ppp$ can be expressed as the sum of a point on a bounded face of $\ppp$ with a direction in which $\ppp$ is unbounded, and since $V(\ppp')\supseteq V(\ppp)$, it follows that in fact $\ppp'+\rec(\ppp)=\ppp$.

Now suppose that $\ppp'$ and $\ppp$ are U-matroid polyhedra such that $\ppp'+\rec(\ppp)=\ppp$. Let $(E,\D,\rho)$ and $(E,\D',\rho')$ be the corresponding U-matroids. 
By Corollary~\ref{cor:recession-sum}, the rank function of the U-matroid polyhedron $\ppp'+\rec(\ppp)$ is the lattice restriction of $\rho'$ to $\D$. Since $\ppp'+\rec(\ppp)=\ppp$, we have $\rho'|_{\D}=\rho$. Then the rank function of $\ppp'$ is an extension of the rank function of $\ppp$. Theorem~\ref{thm:sheared-correspondence} now shows that $\ppp'$ is a sheared polyhedron of $\ppp$.
\end{proof}

For any polyhedron $\qqq\subseteq \ppp$ it is automatic that $\qqq+\rec(\ppp)\subseteq \ppp$ as well, but this containment can easily be strict.  Proposition~\ref{minkmat} says that, actually, for sheared polyhedra, equality holds.  Observe also that the corollary does not depend on the characterization of recession cones of submodular systems (Theorem~\ref{thm:extreme-ray}).

\begin{example}\label{ex:stal-extensions}
    Consider the stalactite $\sss$ of Figure~\ref{fig:stalactite}. The convex hull of the four black points together 0011 and 0101 is the uniform matroid polytope, and hence a sheared polytope of $\sss$. The convex hull of the four black points with \emph{either} of 0011 or 0101 is a sheared polytope contained within $\sss$, but the four black points by themselves do not span a U-matroid polytope and hence do not span a sheared polytope of $\sss$. These three sheared polytopes of $\sss$  correspond to the three matroid lattice extensions of $\rho$, which necessarily give $\rho(4)=1\leq\rho(24),\rho(34)$ and must set $\rho(24)$, $\rho(34)$, or both equal to $2$.
\end{example}

\subsection{Generous extensions} \label{sec:generous}

Let $(E,\D,\rho)$ be a U-matroid with $\D\neq 2^E$.  Since $\D$ is assumed to be accessible, its atoms are all singleton subsets of $E$.  Define $\Atom(\D)=\{a\in E\st\{a\}\in\D\}$.

Suppose that $a\in E\sm\Atom(\D)$.  Let $\D[a]$ be the distributive sublattice of $2^E$ generated by $\D\cup\{\{a\}\}$, i.e.,
\begin{equation} \label{define-Da}
\D[a] = \D\cup\{S\cup\{a\}\st S\in\D\}.
\end{equation}
More generally, for $A\subseteq E$, let $\D[A]$ be the distributive sublattice of $2^E$ generated by $\D\cup \{ \{a\} : a \in A\}$, i.e.,
\begin{equation} \label{define-DA}
\D[A] = \D\cup\{S\cup A'\st S\in\D,\ A'\subseteq A\}.
\end{equation}

\begin{defn} \label{defn:generous-atom}
With the foregoing setup, the \defterm{generous atom extension} of $\rho$ to $\D[a]$ is the function $\rho_a:\D[a]\to\Nn$ defined by
\begin{subnumcases}{\rho_a(S)=}
\rho(S) & if $S\in\D$, \label{genX:1}\\
\rho(S-a) & if $S\not\in\D$ and $\rho(S-a)=\rho(\supn_\D(S))$, \label{genX:2}\\
\rho(S-a)+1 & if $S\not\in\D$ and $\rho(S-a)<\rho(\supn_\D(S))$. \label{genX:3}
\end{subnumcases}
\end{defn}

To explain the term ``generous'', appending $a$ to $S$ increases its rank whenever possible, i.e., except when
the condition $\rho(S-a)=\rho(\supn_\D(S))$
forces $\rho_a(S-a)=\rho_a(S)=\rho_a(T)$.

\begin{remark} \label{rmk:magnanimous}
The results and proofs of \S\ref{sec:sheared-basics} regarding sheared polyhedra and lattice restriction/extension for U-matroids (Definition~\ref{defn:sheared}--Proposition~\ref{minkmat}) remain valid without change if ``U-matroid'' is replaced by ``submodular system'' throughout.  Accordingly, it is natural to ask whether generous extension makes sense for general submodular systems as well.  Indeed, there is a much simpler submodular lattice extension of $\rho$ to $\D'\supseteq\D$, namely the rank function $\rho'(S)=\rho(\supn_\D(S))$, which we call the \defterm{magnanimous extension}.  It is well-defined for any lattice $\D'\supseteq\D$, and proving submodularity is much easier than for generous atom extensions.  On the other hand, the magnanimous extension has the disadvantage that it does not  generally take U-matroids to U-matroids.  (For example, take $\D=\{\0,1,12\}$, $\rho(A)=|A|$, and $\D'=2^E$.)  We will return to it briefly later; see Remark~\ref{rmk:magnanimous2}.
\end{remark}

\begin{thm} \label{thm:Dextend}
The triple $(E,\D[a],\rho_a)$ is a U-matroid (hence a lattice extension of $(E,\D,\rho)$).
\end{thm}

\begin{proof}
Calibration and integrality are immediate from the construction of $\rho_a$.  Let $S,S'\in\D[a]$ with $S'=S\cup\{e\}$.  If $S\in\D$ then either $S'\in\D$, when monotonicity and unit-increase for $\rho_a$ follow from those properties for $\rho$, or $e=a$, when they follow from the definition.  Suppose that $S\notin\D$, so in particular $a\in S$ and $e\neq a$.
Then
\[\rho_a(S')-\rho_a(S)
= \underbrace{\rho_a(S')-\rho_a(S'-a)}_p ~+~ 
\underbrace{\rho_a(S'-a)-\rho_a(S-a)}_q ~-~
\underbrace{\big(\rho_a(S)-\rho_a(S-a)\big)}_r
\]
Each of $p,q,r$ equals 0 or 1, either by monotonicity and unit-increase for $\rho$, or by definition of $\rho_a$.  Therefore, $\rho_a(S')-\rho_a(S)\in\{0,1\}$ unless either $(p,q,r)=(0,0,1)$ or $(p,q,r)=(1,1,0)$.
\begin{itemize}
\item If $p=q=0$, then $\rho_a(S')=\rho_a(S-a)$, and since $S'\supseteq S\supseteq S-a$ all three sets must have equal ranks, so $r=0$, ruling out the possibility $(p,q,r)=(0,0,1)$.
\item If $r=0$, then $\rho_a(S)=\rho(T)$, where $T=\supn_\D(S)$.  In particular $a\in T$, so $T\cup\{e\}=T\cup(S'-a)\in\D$ as well, and
\[\rho(T\cup\{e\})\leq \rho(T)+1=\rho(S-a)+1=\rho_a(S')-p-q+1\leq \rho(T\cup\{e\})-p-q+1\]
which implies $p+q\leq 1$, ruling out the possibility $(p,q,r)=(1,1,0)$.
\end{itemize}
\medskip

Now we prove that $\rho_a$ is submodular.  In all cases, let $S_1,S_2\in\D[a]$, $T_1=\supn_\D(S_1)$ and $T_2=\supn_\D(S_2)$.
\medskip

\noindent\underline{\textbf{Case 1:}  $S_1,S_2\in\D$.}
Then submodularity for $\rho_a$ reduces to submodularity for $\rho$.
\medskip

\noindent\underline{\textbf{Case 2:}  $S_1\in\D$ and $S_2\notin\D$} (w.l.o.g.).
Then $a\in S_2$ and $S_2-a\in\D$.
\medskip

\underline{\textbf{Case 2a:} $\rho(S_2-a)=\rho(T_2)$.} Then
\begin{align*}
\rho_a(S_1)+\rho_a(S_2) 
&= \rho(S_1)+\rho(T_2) \\
&\geq \rho(S_1\cup T_2)+\rho(S_1\cap T_2) && \textup{(by submodularity of $\rho$)}\\
&= \rho_a(S_1\cup T_2)+\rho_a(S_1\cap T_2)\\
&\geq \rho_a(S_1\cup S_2)+\rho_a(S_1\cap S_2) && \textup{(by monotonicity)}.
\end{align*}

\underline{\textbf{Case 2b:} $\rho(S_2-a)<\rho(T_2)$.}  Then
\begin{align*}
\rho_a(S_1)+\rho_a(S_2) 
&= \rho(S_1)+\rho(S_2-a) + 1 \\
&\geq \rho(S_1\cup(S_2-a))+\rho(S_1\cap(S_2-a)) + 1 && \textup{(by submodularity of $\rho$)}\\
&= \rho_a(S_1\cup(S_2-a))+\rho_a(S_1\cap(S_2-a)) + 1\\
&= \rho_a(S_1\cup(S_2-a))+\rho_a((S_1\cap S_2)-a) + 1\\
&= \begin{cases}
\rho_a(S_1\cup S_2)+\rho_a((S_1\cap S_2)-a) + 1 & \textup{ if $a\in S_1$}\\
\rho_a(S_1\cup(S_2-a))+\rho_a(S_1\cap S_2) + 1 & \textup{ if $a\notin S_1$}
\end{cases}\\
&\geq \rho_a(S_1\cup S_2)+\rho_a(S_1\cap S_2) && \textup{(by unit-increase)}.
\end{align*}

\noindent\underline{\textbf{Case 3:} $S_1,S_2\notin\D$.}  Then $a\in S_1\cap S_2$ and $S_1-a,S_2-a\in\D$.
\medskip

\underline{\textbf{Case 3a:} $\rho(S_i-a)=\rho(T_i)$ for both $i=1$ and $i=2$.}  Then
\begin{align*}
\rho_a(S_1)+\rho_a(S_2) 
&= \rho(T_1)+\rho(T_2) \\
&\geq \rho(T_1\cup T_2)+\rho(T_1\cap T_2) && \textup{(by submodularity of $\rho$)}\\
&= \rho_a(T_1\cup T_2)+\rho_a(T_1\cap T_2)\\
&\geq \rho_a(S_1\cup S_2)+\rho_a(S_1\cap S_2) && \textup{(by monotonicity)}.
\end{align*}

\underline{\textbf{Case 3b:} $\rho(S_1-a)=\rho(T_1)$ and $\rho(S_2-a)<\rho(T_2)$.}  Then
\begin{align*}
\rho_a(S_1)+\rho_a(S_2) 
&= \rho(T_1)+\rho(S_2-a) + 1 \\
&\geq \rho(T_1\cup(S_2-a))+\rho(T_1\cap(S_2-a)) + 1 && \textup{(by submodularity of $\rho$)}\\
&= \rho_a(T_1\cup(S_2-a))+\rho_a(T_1\cap(S_2-a)) + 1\\
&= \rho_a(T_1\cup S_2)+\rho_a((T_1\cap S_2)-a) + 1\\
&\geq \rho_a(S_1\cup S_2)+\rho_a((S_1\cap S_2)-a) + 1 && \textup{(by monotonicity)}\\
&\geq \rho_a(S_1\cup S_2)+\rho_a(S_1\cap S_2) && \textup{(by unit-increase)}.
\end{align*}

\underline{\textbf{Case 3c:} $\rho(S_i-a)<\rho(T_i)$ for both $i=1$ and $i=2$.}  Then
\begin{align*}
\rho_a(S_1)+\rho_a(S_2) 
&= \rho(S_1-a)+\rho(S_2-a) + 2 \\
&\geq \rho((S_1-a)\cup(S_2-a))+\rho((S_1-a)\cap(S_2-a)) + 2 && \textup{(by submodularity of $\rho$)}\\
&= \rho_a((S_1\cup S_2)-a)+\rho_a((S_1\cap S_2)-a) + 2\\
&\geq \rho_a(S_1\cup S_2)+\rho_a(S_1\cap S_2) && \textup{(by unit-increase)}. 
\qedhere
\end{align*}
\end{proof}

Intuitively, $\rho_a$ is designed so that appending $a$ to an element of $\D$ increases its rank whenever possible; the proof makes this precise.  Given two real-valued functions $\rho,\phi$ on $\D$, we say that $\phi$ \defterm{dominates} $\rho$ if $\phi(S)\geq\rho(S)$ for every $S\in\D$.  Similarly, if $\rho,\phi$ are submodular functions, we say that the U-matroid $(E,\D,\phi)$ dominates the U-matroid $(E,\D,\rho)$.

\begin{prop} \label{generous-dominates-slightlystronger}
Let $(E,\D,\rho)$ and $(E,\D,\phi)$ be U-matroids such that~$\phi$ dominates~$\rho$.  Then the generous atom extension $\phi_a$ of $\phi$ to $\D[a]$ dominates every lattice extension of $\rho$ to $\D[a]$.
\end{prop}

\begin{proof}
Let $\rho'$ be a lattice extension of $\rho$ to $\D[a]$, and let $S\subseteq E$. If $S\in\D$, then evidently $\rho'(S)=\rho(S)\leq\phi(S)=\phi_a(S)$ by~\eqref{genX:1}.  Accordingly, consider the case that $S\not\in\D$ and so $S-a\in\D$.

Suppose first that $\phi(\supn_\D(S))=\phi(S-a)$. Then by~\eqref{genX:2}, $\phi(S-a)=\phi_a(S)=\phi_a(\supn_\D(S))$.  So
\begin{align*}
\rho'(S)
&\leq\rho'(\supn_\D(S))&&\textup{(by monotonicity)}\\
&\leq\phi_a(\supn_\D(S))&&(\supn_\D(S)\in\D)\\
&=\phi_a(S).
\end{align*}
On the other hand, if $\phi(\supn_\D(S))>\phi(S-a)$, then by \eqref{genX:3}, $\phi(S)=\phi(S-a)+1$. So
\begin{align*}
\phi'(S)
&\leq\rho'(S-a)+1&&\textup{(by unit increase)}\\
&\leq\phi_a(S-a)+1&&(S-a\in\D)\\
&=\phi_a(S).\qedhere
\end{align*}
\end{proof}

\begin{thm}\label{generous-dominates-full}
Let $U=(E,\D,\rho)$ be a U-matroid with $E\sm\Atom(\D)=\{a_1,\dots,a_m\}$. Then:
\begin{enumerate}
\item The iterated generous matroid extension $\hat\rho=\rho_{a_1,\dots,a_m}=(((\rho_{a_1})_{a_2})\dots)_{a_m}$ dominates every other matroid extension of $\rho$ to $2^E$.
\item In particular, $\hat\rho$ is an invariant of $\rho$ that is independent of the order of the $a_i$ and dominates every other matroid extension of $\rho$.
\item For every lattice $\D'$ with $\D\subseteq\D'\subseteq 2^E$, the lattice extension $\hat\rho|_{\D'}$ dominates every other lattice extension of $U$ to $\D'$.
\end{enumerate}
\end{thm}
Accordingly, we say that $\hat\rho|_{\D'}$ is the \defterm{generous extension of $\rho$ to $\D'$}, and $\hat\rho$ is the \defterm{generous matroid extension} of~$\rho$.

\begin{proof}
1.   Let $(E,2^E,\rho')$ be any matroid extension of $\rho$. By Proposition~\ref{generous-dominates-slightlystronger}, $\hat\rho|_{\D[a_1]}$ dominates $\rho'|_{\D[a_1]}$. Then these latter two restrictions fill the hypothesis of the same theorem; applying it again yields that $\hat\rho|_{\D[a_1,a_2]}$ dominates $\rho'|_{\D[a_1,a_2]}$; Proposition~\ref{generous-dominates-slightlystronger} can be thus applied recursively until we get that $\hat\rho$ dominates $\rho'$.
\medskip

2. By assertion~(1), no matter the order of the $a_i$, the iterated generous atom extension $\rho_{a_1,\dots,a_m}$ dominates all matroid extensions of $\rho$. Thus, all such generous extensions must be equal.
\medskip

3. Immediate from assertion~(2).
\end{proof}

A first consequence of the existence of generous extensions is that U-matroid rank functions on $E$ are precisely the restrictions of matroid rank functions to accessible distributive sublattices of $2^E$.

\begin{cor}\label{cor:any-is-restrict}
    U-matroid rank functions on $E$ are precisely the lattice restrictions of matroid rank functions to accessible distributive sublattices of $2^E$.
\end{cor}
\begin{proof}
    It follows from the definitions that the lattice restriction of a matroid rank function is a U-matroid rank function. On the other hand, any U-matroid rank function $(E,\D,\rho)$ is the lattice restriction of its generous matroid extension.
\end{proof}

We next describe the independence and basis polytopes of a generous matroid extension.

\begin{thm} \label{generous-polytopes}
Let $U=(E,\D,\rho)$ be a U-matroid with generous matroid extension $\hat U$. Then
\begin{align} 
\BBB(\hat U) &= \conv(\BBB(U)\cap\{0,1\}^n) = \BBB(U)\cap[0,1]^n, \textup{ and} \label{generous-base-poly}\\
\III(\hat U) &=  \conv(\AAA(U)\cap\{0,1\}^n) = \AAA(U)\cap[0,1]^n. \label{generous-indep-poly}
\end{align}
\end{thm}

\begin{proof}
We prove~\eqref{generous-indep-poly} first.  It is clear that $\conv(\AAA(U)\cap\{0,1\}^n) \subseteq \AAA(U)\cap[0,1]^n$.  Moreover, the vertices of $\III(\hat U)$ are by definition the characteristic vectors of independent sets, which satisfy the defining inequalities~\eqref{define-submod-poly} of~$\AAA(U)$.  It follows that $\III(\hat U)\subseteq\conv(\AAA(U)\cap\{0,1\}^n)$.

Accordingly, it suffices to prove that $\AAA(U)\cap[0,1]^n\subseteq\III(\hat U)$.  The polytope $\III(\hat U)$ is defined by the inequalities $\xx(i)\geq0$ for all $i\in[n]$ and $\xx(A)\leq\hat\rho(A)$ for all $A\subseteq[n]$ \cite{Edmonds}.  By Theorem~\ref{generous-dominates-full}, it suffices to show that if $\rho_a$ is any generous atom extension of $\rho$ to $\D[a]$, then $\xx(S)\leq\rho_a(S)$ for all $\xx\in\BBB(U)\cap [0,1]^E$ and all $S\in\D[a]\sm\D$.  Note that $S-a\in\D$ in this case.  By Definition~\ref{defn:generous-atom}, there are two cases. First, if $\rho(S-a)=\rho(\supn_\D(S))$, then
    \begin{align*}
        \xx(S)&\leq\xx(\supn_\D(S)) && \textup{(because $x_i\geq 0$ for all $i$)}\\
        &\leq \rho(\supn_\D(S)) && \textup{(because $\xx\in\BBB(U)$)}\\
        &=\rho_a(S) && \textup{(by~\eqref{genX:2}).}
    \end{align*}
Second, if $\rho(S-a)<\rho(\supn_\D(S))$, then
    \begin{align*}
        \xx(S)&\leq\xx(S-a)+1 && \textup{(because $x_a\leq 1$)}\\
        &\leq\rho(S-a)+1 && \textup{(because $\xx\in\BBB(U)$)}\\
        &=\rho_a(S) && \textup{(by~\eqref{genX:3}).}
    \end{align*}
We conclude that $\xx(S)\leq\rho_a(S)$ as desired, completing the proof of~\eqref{generous-indep-poly}.  Now~\eqref{generous-base-poly} follows by intersecting with the hyperplane $\xx(E)=\rho(E)$.
\end{proof}

Theorem~\ref{thm:sheared-correspondence} shows that sheared polyhedra of a U-matroid polyhedron and U-matroid extensions of its rank function are cryptomorphic. We use this to translate Corollary~\ref{cor:any-is-restrict} to the language of polyhedra.

\begin{cor}
    The U-matroid polyhedra over an accessible distributive lattice $\D$ are precisely the Minkowski sums $\BBB(M)+\rec(\D)$, where $M$ is a matroid on $E$.
\end{cor}
\begin{proof}
    It is immediate that a Minkowski sum of a matroid polytope with $\rec(\D)$ is a U-matroid polyhedron over $D$. On the other hand, let $U=(E,\D,\rho)$ be a U-matroid with base polyhedron $\ppp=\BBB(U)$. Let $\rho'$ be any matroid extension of $\rho$ (in particular, we may take $\rho$ to be the generous matroid extension, which always exists). Let $\ppp'$ be the matroid polytope of $\rho'$.
    Then Proposition~\ref{minkmat} shows that $\ppp=\ppp'+\rec(\ppp)$.
\end{proof}

\begin{remark}\label{remk:non-gen-ext}
There is not in general a unique \emph{smallest} sheared polytope within a given U-matroid polyhedron.  Consider the stalactite of Figure~\ref{fig:stalactite}. As explained in Example~\ref{ex:stal-extensions}, the convex hull of the four black points together with either of 0011 or 0011 is a sheared polytope of the stalactite, but the four black points by themselves do not span a matroid polytope and hence do not form a sheared polytope.
\end{remark}

\begin{problem} \label{prob:rank-generous}
Determine a formula for the rank function $\hat\rho$ of the generous matroid extension, in terms of $\rho$.
\end{problem}

\begin{remark} \label{generous-formula-doesnt-work}
It is tempting to conjecture the formula
\[\hat\rho(S)=\min(\rho(\supn_\D(S)),\ \rho(\infn_\D(S))+e_\D(S))\]
where $e_\D(S)=\#\{a\in S\sm \infn_\D(S)\ |\ \rho(\supn_\D(\infn_\D(S)\cup a))>\rho(\infn_\D(S))\}$.
This formula for $\hat\rho(S)$ works in many cases, but fails in general.  For example, let $\P$ be the poset on $[6]$ with relations $1<4,2<5,3<6$, let $\D=J(\P)$, and let $\rho$ be the submodular function given as follows (abbreviating, e.g., $\{2,3,5,6\}$ by 2356):
\begin{align*}
\rho^{-1}(0) &= \{\0\}, & \rho^{-1}(2) &= \{14,12,125,13,136,123,1235,1236,12356\},\\
\rho^{-1}(1) &= \{1,2,25,3,36,23,235,236,2356\}, & \rho^{-1}(3) &= \{124,1245,134,1346,1234,12345,12346,123456\}.
\end{align*}
Then the function $\hat\rho$ as defined above is not submodular --- e.g., $\hat\rho(4)=\hat\rho(56)=1$ and $\hat\rho(456)=3$ --- hence cannot be the generous matroid extension of $\rho$. (We do not claim that this counterexample minimizes the size of $E$.)
\end{remark}

\begin{remark}
Suppose that $(E,\D',\rho')$ is a lattice extension of $(E,\D,\rho)$.  What is the relationship between the normal fans of the corresponding U-polyhedra $\ppp'$ and $\ppp$?  Evidently $|\N(\ppp')|\supset|\N(\ppp)|$.  Moreover, every bounded face $\qqq\subset\ppp$ is also a face of $\ppp'$, with $N_{\ppp'}(\qqq)\supseteq N_\ppp(\qqq)$.  (To see this, observe that every vertex of $\qqq$ is a vertex of $\ppp'$, so $\qqq\subseteq\ppp'$, and for any linear functional $\ell\in N_\ppp(F)$,  the face of $\ppp'$ minimized by $\ell$ is $F\cap \ppp'=F$.)  On the other hand, if $\qqq$ is an \textit{unbounded} face of $\ppp$, then many different things can happen: $\qqq\cap\ppp'$ can equal $\qqq$, or become bounded, or drop in dimension, or even coincide with a smaller face of $\ppp$.
\end{remark}

We now describe how the operations of restriction and generous extension affect closure operators and lattices of flats of U-matroids.

\begin{thm}\label{thm:flat-latticerestrict}
Let $U=(E,\D,\rho)$ be a U-matroid with closure operator $\cl$ and lattice of flats $\L$.  Let $\D'\subseteq\D$ be an accessible distributive sublattice, and let $\L'$ be the lattice of flats of the restriction $U_{\D'}=(E,\D',\rho|_{\D'})$.  Then
\[\L'=\{\infn_{\D'}(A)\ |\ A\in\L\}.\]
\end{thm}

\begin{proof}
($\supseteq$) Let $A\in\L$ and $A'=\infn_{\D'}(A)$.  Suppose that there exists some $e\in E\sm A'$ such that $A'\cup e\in\D'\subseteq\D$.  Then $A\cup e=A\cup(A'\cup e)$ belongs to $\D$ as well.  Moreover, $e\notin A$ (otherwise $A'\cup e$ is an element of $\D'$ larger than $A'$), so $\rho(A\cup e)>\rho(A)$ by~\eqref{flats-from-rank-no-sup}, so $\rho(A'\cup e)>\rho(A')$ by submodularity.  It follows by~\eqref{flats-from-rank-no-sup} that $A'\in\L'$.

($\subseteq$) Let $A'\in\L'$, and let $A=\cl(A')$; in particular $A\in\L$ and $\rho(A)=\rho(A')$.  Evidently $A'\subseteq\infn_{\D'}(A)$; we claim that equality holds.  Indeed, if $e\in\infn_{\D'}(A)$, then
$\supn_{\D'}(e)\subseteq\supn_{\D'}(\infn_{\D'}(A))=\infn_{\D'}(A)\subseteq A$, so $A'\subseteq\supn_{\D'}(A'\cup e) = \supn_{\D'}(A')\cup\supn_{\D'}(e)= A'\cup\supn_{\D'}(e)\subseteq A$.  It follows that $\rho(A')=\rho(\supn_{\D'}(A'\cup e))=\rho(A)$, but since $A'$ is a flat of $\D'$, it follows by~\eqref{flats-from-rank} that $e\in A'$, completing the proof of the claim.
\end{proof}

\begin{cor}\label{thm:closure-latticerestrict}
    Let $\cl_U$ be a U-matroid closure operator on $(E,\D)$, and let $\D'\subseteq\D$ be an accessible distributive sublattice. Then the lattice restriction $U'=U|_{\D'}$ has closure operator
    \[\clo[U']{A}=\infn_\D(\clo[U]{A}).\]
\end{cor}

We now look at the generous extension in these settings, again starting with the lattice of flats.

\begin{thm}\label{thm:flat-rank-genext}
Let $U=(E,\D,\rho)$ be a U-matroid with lattice of flats $\L$.  Let $a\in E$ such that $\{a\}\notin\D$ and $E\sm\{a\}\in\D$. Let $U'=(E,\D',\rho)$ be the generous atom extension of $U$ to $\D'=\D[a]$.  Then the lattice of flats of $U'$ is $\L'=\L\cup\K$, where
\begin{align*}
\K&=\{A\cup\{a\}\st A\in\L \text{ and } \rho(\supn_\D(A\cup a))\geq\rho(A)+2\}\\
&= \{A\cup\{a\}\st A\in\L \text{ and } a\not\in B \text{ for all } B\in\L \text{ covering } A\}.
\end{align*}
\end{thm}

\begin{proof}
($\L'\supseteq\L\cup\K$) We first prove $\L'\supseteq\L$. Suppose $A\in\L$. If $a\in A$, then any $S\in\D'$ containing $A$ must be an element of $\D$, and since $A\in\L$ we must have $\rho(S)>\rho(A)$; it follows that $A$ is a flat of $U'$. If $a\not\in A$, then in order to show $A\in\L'$ it suffices to check that $\rho(\supn_{\D'}(A\cup a))>\rho(A)$.  Indeed, $\rho(\supn_{\D}(A\cup a))>\rho(A)$ by~\eqref{flats-from-rank-alt},
which by \eqref{genX:3} implies $\rho(A\cup a)=\rho(\supn_{\D'}(A\cup a))=\rho(A)+1$, as desired.  We conclude that $\L'\supseteq\L$.

We now show that $\L'\supseteq\K$. Take $A\in\L$ such that $\rho(\supn_\D(A\cup a))\geq\rho(A)+2$; in particular $a\not\in A$. If $A\cup a\notin\L'$ then, by definition, there is some $e\in\D\sm (A\cup a)$ such that $(A\cup a)\cup e\in\D'$ and $\rho((A\cup a)\cup e)=\rho(A\cup a)$.  Since $E\sm\{a\}\in\D$ and $A\cup a\cup e\in\D'$, the set $A\cup e=(E\sm\{a\})\cap(A\cup a\cup e)$ belongs to $\D'$, hence (since it does not contain $a$) to $\D$. Then $\rho(\supn_\D(A\cup e\cup a))\geq\rho(\supn_\D(A\cup a))\geq\rho(A)+2>\rho(A\cup e)$. If $A\cup e\cup a\in\D$ then this inequality forces $\rho(A\cup a\cup e)=\rho(A\cup e)+1$, while if $A\cup e\cup a\not\in\D$ then by~\eqref{genX:3} we have $\rho(A\cup a\cup e)=\rho(A\cup e)+1=\rho(A\cup a)+1$. In each case, $\rho(A\cup a\cup e)=\rho(A\cup e)+1$, which contradicts our choice of $e$. We conclude that $A\cup a\in\L'$, so $\L'\supseteq\K$ as desired.
\medskip

($\L'\subseteq\L\cup\K$) Let $B\in\L'$.
First, suppose $B\in\D$. Since $B$ is a flat of $U'$, $\rho(B\cup e)>\rho(B)$ for all $e\in E\sm B$ such that $B\cup e\in\D'$ --- in particular, for all $e\in E\sm B$ such that $B\cup e\in\D$.  Thus $B$ is a flat of $U$, i.e., $B\in\L$.

Second, suppose $B\not\in\D$, so that $B=A\cup a$ for some $A\in\D$.  We claim that $A\in\L$. Take $e\in E\sm A$ with $A\cup e\in\D$. In particular, $e\neq a$, so $e\in E\sm B$ and $B\cup e=B\cup(A\cup e)\in\D'$. Then $\rho(B\cup e)>\rho(B)$ (since $B\in\L'$), so $\rho(A\cup e)>\rho(A)$  by submodularity of $\D'$, proving the claim. Then $\rho(\supn_\D(A\cup a))>\rho(A)$ and $\rho(A\cup  a)=\rho(A)+1$ by~\eqref{genX:3}. Since $A\cup a=B$ is a flat of $U'$ and $\supn_\D(B)\supsetneq B$, we have $\rho(\supn_\D(B))\geq\rho(B)+1\geq\rho(A)+2$. It follows that $B=A\cup a\in\K$, as desired.
\end{proof}

\begin{remark} \label{rmk:magnanimous2}
Here we return briefly to the operation of magnanimous extension, described in Remark~\ref{rmk:magnanimous}.  Using the monotonicity property~\eqref{monotonicity}, one can show easily that the magnanimous extension of $S$ to a polymatroid $\hat S=(E,2^E,\hat\rho)$ dominates all other polymatroid extensions of~$S$, and that its base polytope is $\BBB(\hat S)=\BBB(S)\cap\Rr_{\geq0}^E$.  As a corollary, every submodular system is a lattice restriction of a polymatroid; equivalently, every submodular system polyhedron (i.e., every extended generalized permutahedron) is the Minkowski sum of a (bounded) generalized permutahedron with a cone.
\end{remark}

\section{Basis systems and the pseudo-independence complex: Generalization of Bj\"orner's and Gale's theorems} \label{sec:basis}

Let $U=(E,\D,\rho)$ be a U-matroid with basis system $\B=\B(U)$.  The \defterm{pseudo-independence complex} $\Delta(U)$ is the simplicial complex on $E$ generated by $\B$.  When $U$ is a matroid, $\Delta(U)$ corresponds to the vertex set of the independence polytope $\III(U)=\AAA(U)\cap[0,1]^n$ (see Theorem~\ref{generous-polytopes}), but for general U-matroids the connection between $\Delta(U)$ and~$\AAA(U)$ can break down.  For instance, if $U$ is the U-matroid of the stalactite (see Example~\ref{ex:stalactite}), then $\AAA(M)\cap[0,1]^n$ has 11 vertices but $\Delta(M)$ has only 9 faces.
On the other hand, pseudo-independence complexes of U-matroids share a crucial property with independence complexes of matroids: they are shellable.  This property is both significant in its own right and useful for characterizing U-matroid basis systems.

\begin{thm} \label{thm:shelling}
The pseudo-independence complex $\Delta(U)$ of every U-matroid $U$ is shellable.
\end{thm}

\begin{proof}
We will make use of the following result.  For a pure polyhedral cell complex~$K$, let~$\F(K)$ denote the set of facets of~$K$, let $G_K$ denote the dual graph on~$\F(K)$ (in which two facets are adjacent if they intersect in a face of codimension~1), and let $H_X$ denote the intersection graph on~$\F(K)$ (in which two facets are adjacent if they intersect in a nonempty face).  Heaton and Samper \cite[Thm.~3]{HeatonSamper} proved that, for a pure polytopal complex~$K$ and a pure simplicial complex~$\Delta$, if there exists a bijection $\phi:\F(K)\to\F(\Delta)$ that induces a bijection $G_K\to G_\Delta$ and an embedding $H_K\to H_\Delta$, then $\phi$ maps every shelling of~$K$ to a shelling of~$\Delta$.  Our plan is to apply this result with $\Delta=\Delta(U)$.  The problem is to construct a suitable polyhedral complex~$K$.

Let $\ppp\subset\mathbb{R}^n$ be a polyhedron with at least one vertex, and let $\ccc=\rec(\ppp)$. Let $\ell\in (-\ccc^*)^\circ$, i.e., $\ell\cdot\vv>0$ for all $\vv\in \ccc$.  For $a\in\Rr$, define a polyhedron by
\[\qqq = \qqq(a) = \{\xx\in\ppp\st \ell(\xx)\leq a\}.\]
We claim that $\qqq$ is in fact bounded, i.e., a polytope.  Otherwise, $\qqq$ contains a ray whose direction vector $\vv$ belongs to $\ccc$, but then $\ell\cdot\vv>0$, so $\ell$ increases without bound on the ray and hence on $\qqq$, contradicting the definition of $\qqq$.  Furthermore, observe that we may choose $a$ sufficiently large so that
\begin{enumerate}
\item every facet of $\ppp$ gives rise to a facet of $\qqq$;
\item $\qqq$ has one additional facet, defined by $\ell(\xx)=a$;
\item $V(\ppp)\subset V(\qqq)$, and every edge in $\qqq$ between two vertices of $\ppp$ is also an edge of $\ppp$; and 
\item if $\xx\in V(\qqq)$, then $\ell(\xx)<a$ if and only if $\xx\in V(\ppp)$.
\end{enumerate}
An example of this ``cutting-off'' construction is shown in Figure~\ref{fig:cutoff}.

\begin{figure}[htb]
\begin{center}
\begin{tikzpicture}[scale=0.75]
\newcommand{\scl}{0.7}
\newcommand{\xiooi}{0*\scl}	\newcommand{\yiooi}{6*\scl}
\newcommand{\xooii}{3*\scl}	\newcommand{\yooii}{4*\scl}
\newcommand{\xioio}{-1*\scl}	\newcommand{\yioio}{4*\scl}
\newcommand{\xoioi}{1*\scl}	\newcommand{\yoioi}{2*\scl}
\newcommand{\xiioo}{-3*\scl}	\newcommand{\yiioo}{2*\scl}
\newcommand{\xoiio}{0*\scl}	\newcommand{\yoiio}{0*\scl}
\coordinate (oiio) at (\xoiio,\yoiio);
\coordinate (ioio) at (\xioio,\yioio);
\coordinate (iioo) at (\xiioo,\yiioo);
\coordinate (iooi) at (\xiooi,\yiooi);
\fill[blue!10] (oiio)--(iioo)--(iooi)--(7+\xiooi/2,\yiooi)--(7+\xooii/2,\yooii)--(7+\xoiio/2,\yoiio)--cycle;
\draw[thick,red, fill=red!20] (7+\xiooi,1+\yiooi)--(7+\xooii,1+\yooii)--(7+\xoiio,-1+\yoiio)--(7+\xiioo,-1+\yiioo)--cycle;
\foreach \coo in {(oiio),(iooi),(iioo)} \draw[dashed] (ioio)--\coo;
\draw[thick] (oiio)--(iioo)--(iooi);
\draw[dashed] (ioio)--(11+\xioio,\yioio);
\draw[thick] (iooi)--(11+\xiooi,\yiooi);
\draw[thick] (oiio)--(11+\xoiio,\yoiio);
\draw[thick] (iioo)--(11+\xiioo,\yiioo);
\draw[fill=black] (\xoiio,\yoiio) circle(.1);
\draw[fill=black] (\xioio,\yioio) circle(.1);
\draw[fill=black] (\xiioo,\yiioo) circle(.1);
\draw[fill=black] (\xiooi,\yiooi) circle(.1);
\draw[very thick,green!50!black,dashed] (7+\xiooi/2,\yiooi)--(7+\xooii/2,\yooii)--(7+\xoiio/2,\yoiio);
\draw[very thick,green!50!black] (7+\xiooi/2,\yiooi)--(7+\xiioo/2,\yiioo)--(7+\xoiio/2,\yoiio);
\draw[very thick,fill=green] (7+\xiooi/2,\yiooi) circle(.1);
\draw[very thick,fill=green] (7+\xooii/2,\yooii) circle(.1);
\draw[very thick,fill=green] (7+\xiioo/2,\yiioo) circle(.1);
\draw[very thick,fill=green] (7+\xoiio/2,\yoiio) circle(.1);

\draw[very thick,->,red] (\xoioi,\yoioi+.75)--(\xoioi+2,\yoioi+.75);
\node[red] at (\xoioi+1,\yoioi+.4) {$\ell$};
\node[red] at (\xoioi+9,\yoioi+.75) {$\ell(\xx)=a$};
\node at (2,.5) {\large$\qqq$};
\node at (12,2) {\large$\ppp$};

\end{tikzpicture}
\caption{``Cutting off'' a polyhedron $\ppp$ with a hyperplane $\ell(x)=a$ to produce a polytope $\qqq$.\label{fig:cutoff}}
\end{center}
\end{figure}

We now apply this construction to the base polytope $\ppp=\BBB(U)$.  (Note that we do not claim that $\qqq$ is a generalized permutahedron.)
Consider the polar dual $\qqq^\vee$ of $\qqq$.  The facets of $\qqq^\vee$ corresponding to the vertices of $\ppp$ form a pure polytopal complex $K\subset\bd \qqq^\vee$ such that $G_K\isom G_{\Delta(U)}$.  Now consider a line shelling of $\qqq^\vee$ as described in \cite[\S8.2]{Ziegler}, moving in the direction of decrease of $\ell$ (perturbing $\ell$ if necessary so as to meet the facets of $\qqq^\vee$ at distinct points).  The facets that become visible in the first phase of the shelling are precisely those in $K$.  In particular $K$ is an initial segment of a shelling, so it is itself shellable.  By the result of Heaton and Samper stated earlier \cite[Thm.~3]{HeatonSamper}, ordering the vertices of $\ppp$ in decreasing order of the value of $\ell$ gives a shelling order on~$\Delta$.
\end{proof}

In particular, let $\sigma$ be any linear extension of $\P$, so that the braid cone $\ccc_\sigma$ lies in $\N(\P)$.  Then any linear functional $\ell\in -\ccc_\sigma^\circ$ defines a shelling of $\Delta(U)$.  By varying $\ell$ within the same braid cone, we can obtain different shelling orders, including the lexicographic and reverse-lexicographic orders on $\binom{E}{r}$ corresponding to~$\sigma$.
(By comparison, a theorem of Bj\"orner \cite[Thm.~7.3.4]{Bjorner-shellability} states that a pure complex is a matroid independence complex if and only if \textit{every} ordering on the facets induced lexicographically by a total ordering of the vertices is a shelling order.)

\begin{example} \label{shell-stalactite}
Consider the stalactite~$\sss$ of Example~\ref{ex:stalactite}, with ground set $E=[4]$ and $\B=\{12,13,14,23\}$.  Its characteristic poset contains the single relation $1<4$.  Its pseudo-independence complex is just the graph with edges $\B$, which is shellable but not a matroid complex.
Every linear extension $\sigma$ of $\P$ gives rise lexicographically to a shelling order of $\Delta$; for example, $\sigma=3<1<2<4$  gives rise to the shelling order $13,12,14,24$.  In fact, all twelve linear extensions give rise to different shelling orders.  (As it happens, there are only two linear orders on $E$ that do \textit{not} give rise to shelling orders, namely $4<3<2<1$ and $4<2<3<1$.)
\end{example}

We next characterize basis systems of U-matroids.  Specifically, given a partial order $\P$ on~$E$ and a nonempty set family $\B\subseteq\binom{E}{r}$, we will give necessary and sufficient conditions for $\B$ to be the basis system of a U-matroid with characteristic poset $\P$.
For $\sigma\in \Le(\P)$, let $f_\B(\sigma)$ denote the lexicographically smallest element of $\B$ with respect to $\sigma$
(see \S\ref{subsec:lattices}).

\begin{thm} \label{thm:basis-theorem}
Let $E=[n]$ and $0\leq r\leq n$.  Fix a poset $\P$ on $E$, and let $\B\subseteq\binom{E}{r}$.
Then the following are equivalent:
\begin{enumerate}[label=(\alph*)]
\item\label{Umatroid} There exists a U-matroid $(E,J(\P),\rho)$ with basis system $\B$.
\item\label{very-optimistic} The following conditions hold:
    \begin{enumerate}[label=(B\arabic*)]
        \item\label{VO:1} The function $f_\B\st\Le(\P)\to\B$ is surjective;
        \item\label{VO:2} For every $\sigma\in \Le(\P)$, the set $\B$ has a unique minimum element with respect to $\lgale^\sigma$.
    \end{enumerate}
\item\label{optimistic} The following conditions hold:
    \begin{enumerate}[label=(B\arabic*$'$)]
        \item\label{O:1} The function $f_\B\st\Le(\P)\to\B$ is surjective;
        \item\label{O:2} For every $\sigma\in \Le(\P)$, the order $\llex^\sigma$ is a shelling order of the simplicial complex with facets~$\B$.   
    \end{enumerate}
\end{enumerate}
\end{thm}

\begin{proof}
\ref{Umatroid}$\implies$\ref{optimistic}: This implication follows from Theorem~\ref{thm:shelling} and the discussion thereafter.
\medskip

\ref{optimistic}$\implies$\ref{very-optimistic}: Assume that $\B$ satisfies conditions~\ref{O:1} and~\ref{O:2}.  Condition~\ref{VO:1} is identical to~\ref{O:1}.
To show~\ref{VO:2}, let $\sigma\in \Le(\P)$.  Order the bases lexicographically with respect to $\sigma$ as $B_1,\dots,B_k$; by~\ref{O:2} this is a shelling order.
By the definition of shelling, for every $j>1$, there exists $i\in[1,j-1]$ such that $|B_i\cap B_j|=r-1$.  Since $B_i\llex B_j$, it must be of the form $B_j\sm\{y\}\cup\{x\}$, where $x<y$.  In particular, $B_i\lgale B_j$ as well.  By induction, it follows that $B_1\lgale B_j$ for all $j>1$.
\medskip

\ref{very-optimistic}$\implies$\ref{Umatroid}:
Assume that $\B$ satisfies~\ref{VO:1} and~\ref{VO:2}.  Define a function $\rho: J(\P) \to \Nn$ by
\[\rho(A)=\max_{B\in \B}\{|B\cap A|\}.\]
We claim that $(E,J(\P),\rho)$ is a U-matroid.
It is immediate that $\rho$ satisfies the calibration and monotonicity conditions of Definition~\ref{def:submodsyst},as well as the integrality and unit-increase conditions of Definition~\ref{def:U-matroid}.  It remains only to check submodularity, for which we use the following fact.

\textbf{Claim:}
Let $A\in J(\P)$, let $\sigma \in \Le(\P)$ be a total order with $A$ as an initial segment (i.e., if $a\in A$ and $z\in E\sm A$, then $a<_\sigma z$), and let $B=f_\B(\sigma)$. Then $\rho(A)= |A\cap B|$.

To prove the claim, for any $B'\in\B$, the set $A\cap B'$ consists of the $k$ smallest elements of $B'$, for some $k$ (since $A$ is an order ideal).  But by~\ref{VO:2} $A\cap B$ must then contain the $k$ smallest elements of $B$, and it follows that $|A\cap B|\geq|A\cap B'|$.

Now, let $A, A'\in J(\P)$. Consider a linear extension $\sigma$ of $\P$ that contains $A\cap A'$, $A$ and $A\cup A'$ as initial segments; note that such a $\sigma$ exists because $A\cap A'\subseteq A\subseteq A\cup A'$.  
Let $B=f_\B(\sigma)$; then the claim implies that
\[\rho(A\cup A')+\rho(A\cap A')-\rho(A)
= |B\cap(A\cup A')|+|B\cap(A\cap A')|-|B\cap A| 
= |B\cap A'|
\leq \rho(A')\]
which is the submodular inequality.

Now we show that the basis system of $U=(E,\D,\rho)$ is in fact $\B$ itself.  By the recipe~\eqref{basis-from-linext} for bases in terms of linear extensions of $\P$, every basis $B$ satisfies $\rho(B)=|B|=r$, which implies $X\in\B$.
Conversely, suppose $B\in\B$. Then $B=f_\B(\sigma)$ for some $\sigma\in \Le(\P)$ by~\ref{VO:1}.  Let $A=B_\rho(\sigma)=\{a_1<\cdots<a_r\}$.
Suppose that $A'\in\binom{E}{r}$ and $A'\llex^\sigma A$, say $A'=\{a'_1<\cdots<a'_r\}$ with $a'_i=a_i$ for $i<k$ and $a'_k<_\sigma a_k$. Then $\rho(\{a'_1,\dots,a'_k\})<k$ by by Theorem~\ref{thm:submod-vertex}, and the unit-increase property then implies $\rho(A')<r$, i.e., $A'\notin\B$.  It follows that $A=f_\B(\sigma)=B$ as desired.
\end{proof}

In the special case that $\P$ is an antichain, the equivalence of \ref{Umatroid} and \ref{optimistic} reduces to Bj\"orner's theorem \cite[Thm.~7.3.4]{Bjorner-shellability} that a pure simplicial complex is matroidal if and only if every total order on the vertex set lexicographically induces a shelling order on the facets.

\section{Duality for U-matroids} \label{sec:dual}

If $\B$ is a matroid basis system on ground set $E$, its dual is the matroid basis system $\B^*=\{E\sm B\st B\in\B\}$.
The definition of duality for U-matroids is conveniently expressed geometrically.

\begin{defn} \label{defn:Udual}
Let $\ppp\subset\Rr^E$ be a U-matroid polyhedron.  The \defterm{U-matroid dual}\footnote{Not to be confused with the polar dual $\ppp^\vee$ of $\ppp$.} of~$\ppp$ is
\[\ppp^*=\{(1,1,\dots,1)-\xx \st \xx\in\ppp\}.\]
Equivalently, $\ppp^*$ is the image of~$\ppp$ under reflection through the point $(\frac12,\frac12,\dots,\frac12)$.
\end{defn}

Indeed, the vertices of $\ppp^*$ are (0,1)-vectors (they are the complements of the vertices of $\ppp$), and its 1-dimensional faces are parallel to the corresponding faces of $\ppp$, so it is a U-matroid polyhedron.  Moreover, $\rec(\ppp^*)=-\rec(\ppp)$, so if the corresponding U-matroids are $U=(E,\P,\rho)$ and $U^*=(E,\P^*,\rho^*)$, then $\P^*$ is indeed the poset dual of $\P$, and the characteristic lattice of $U$ is $\D^*=\{\bar A\st A\in\D\}$, where $\bar A=E\sm A$.
 Combining Definition~\ref{defn:Udual} with~\eqref{define-base-poly}, one can show (details omitted) that the dual rank function $\rho^*:\D^*\to\Nn$ is given by
$\rho^*(Z) = \rho(\bar Z)+|Z|-\rho(E)$, just as it is for matroids \cite[Prop.~2.1.9]{Oxley}.

This construction is comparable to the dual of a polymatroid \cite[pp.~29--30]{Fujishige} or general submodular systems \cite[pp.~36--37]{Fujishige}.  (However, duality is not canonical for polymatroids, since in general multiple vectors can play the role of $(1,1,\dots,1)$, and the dual of a submodular system is a \textit{super}modular system.)

While our definition of duality is geometric, it is worth observing that duality of basis systems can be understood purely combinatorially as well.

\begin{prop}
Suppose that $\B$ is the basis system of some U-matroid with characteristic poset $\P$.
Then $\B^*=\{\bar B\st B\in\B\}$ is the basis system of a U-matroid with characteristic poset $\P^*$.
\end{prop}
\begin{proof}
By hypothesis, the pair $(\P,\B)$ satisfies conditions~\ref{VO:1} and~\ref{VO:2} of Theorem~\ref{thm:basis-theorem}.  We will show that the pair $(\P^*,\B^*)$ does too.

First, in general, $A\legale^\sigma A'\iff\bar A\legale^{\sigma^*} \bar{A'}$.  Therefore, for $\sigma^*\in\Le(\P^*)$ (i.e., for $\sigma\in\Le(\P)$), if $B$ is the unique $\legale^\sigma$-minimum of $\B$, then $\bar{B}$ is the unique $\legale^{\sigma^*}$-minimum of $\B^*$.  This establishes~\ref{VO:2}.

Now let $\bar B\in\B^*$.  By~\ref{VO:1}, there exists some $\sigma\in\Le(\P)$ such that $f_\B(\sigma)=B$.  Then $B$ is $\legale^\sigma$-minimal in~$\B$, and by~\ref{VO:2}, it is the unique such minimum.  Thus $\bar B$ is the unique $\legale^{\sigma^*}$-minimum of $\B^*$ and so $\bar B=f_{\B^*}(\sigma^*)$ (since lexicographic order is a linear extension of Gale order), implying~\ref{VO:1} for $(\P^*,\B^*)$.
\end{proof}

Basis systems also behave well with respect to U-matroid restriction:

\begin{prop}
Let $U=(E,\D,\rho)$ be a U-matroid with basis system $\B$.  Let $\D'\subseteq\D$ be a distributive lattice, and let $\P=\Irr(\D)$ and $\P'=\Irr(\D')$ (so in particular $\Le(\P)\supseteq\Le(\P')$). 
As before, let $f_\B(\sigma)$ denote the lexicographically smallest element of $\B$ with respect to $\sigma$.
Then the basis system $\B'$ of the lattice restriction $U'=U|_{\D'}$ is $f_\B(\Le(\P'))$.
\end{prop}
\begin{proof}
If $U$ is a matroid, then $\Le(\P)=\Ord(E)$, and the conclusion follows from Theorem~\ref{thm:submod-vertex}.

For the general case, let $\hat U$ be the generous matroid extension of $U$ and let $\hat\B=\B(\hat U)$.  Then $U=\hat U|_{\D}$, so by the first case $\B=f_{\hat\B}(\Le(\P))$. But $U'=U|_{\D'}=(\hat U|_{\D})|_{\D'}=\hat U|_{\D'}$, so again by the first case $\B'=f_{\hat\B}(\Le(\P'))$.  Now $\Le(\P')\subseteq \Le(\P)$, and $f_{\hat\B}$ and $f_{\B}$ are equal on $\Le(\P)$, so
$\B'=f_{\B}(\Le(\P'))$ as desired.
\end{proof}

\begin{cor}
The U-matroid basis systems on ground set $E$ are exactly the sets of the form $\B\cap\D$, where $\B$ is a matroid basis system on $E$ and $\D$ is an accessible distributive lattice on $E$.
\end{cor}

\begin{qn}
Is there a notion of basis exchange for U-matroids that generalizes matroid basis exchange?
\end{qn}

\section{U-matroids, poset matroids, and subspace arrangements} \label{sec:PMatroid}

In this section, we discuss how U-matroids, and specifically the operation of generous extension, can be used to study subspace arrangements.  In so doing, we describe the connections between U-matroids and two other combinatorial constructions in the literature: the \textit{poset matroids} of Barnabei, Nicoletti, and Pezzoli~\cite{Barnabei} and the \textit{multisymmetric matroids} of Crowley, Huh, Larson, Simpson and Wang~\cite{Crowley}.  As we observe, poset matroids are in fact precisely those U-matroids for which every basis belongs to the characteristic lattice.  

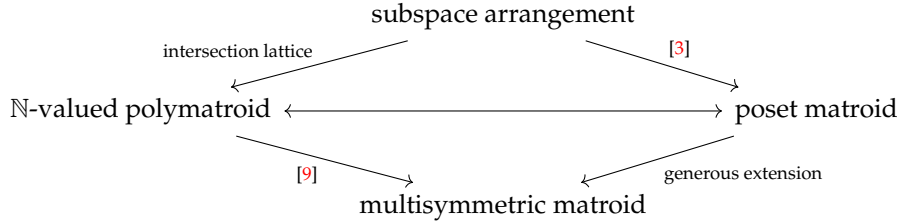
\begin{figure}[ht]
\[\begin{tikzcd}
& \text{subspace arrangement} \arrow[dl, "\text{intersection lattice}"'] \arrow[dr, "\text{\cite{Barnabei}}"]\\
\text{$\Nn$-valued polymatroid}  \arrow[dr, "\text{\cite{Crowley}}"'] \arrow[rr, leftrightarrow] && \text{poset matroid}  \arrow[dl, "\text{generous extension}"]\\
& \text{multisymmetric matroid}
\end{tikzcd}\]
\caption{Combinatorial models of subspace arrangements.\label{fig:mat-arr}}
\end{figure}

A subspace arrangement is a list $\XX=(X_1,\dots,X_m)$ of linear subspaces of a vector space $V\isom\fld^d$; typically the codimensions of the $X_i$ are specified in advance.  A basic combinatorial invariant of an arrangement $\XX$ is its rank function $\rk_\XX$, which gives the codimension of every intersection of spaces in~$\XX$. It is easy to check that $\rk_\XX$ is a polymatroid rank function.
Barnabei et al.\ constructed a poset matroid from $\XX$ whose characteristic lattice is a product of chains, and Crowley et al. constructed a canonical multisymmetric matroid corresponding to an arrangement obtained by replacing each space in $\XX$ by a generic intersection of hyperplanes.  We show that all three of these combinatorial invariants in fact encode the same information (which, by the way, is known to determine the cohomology groups of the complement of a real arrangement, but not the cohomology ring --- see \cite{Bjorner-subspace}).  On the other hand, our construction of the generous atom extension has concrete geometric meaning: it corresponds to replacing one subspace $X$ in $\XX$ with a pair $(X',H)$ of subspaces containing $X$, with $\codim X'=\codim X-1$ and $\codim H=1$.
When iterated, this operation yields the multisymmetric matroid construction of Crowley et al., so that construction is seen to be equivalent to the generous matroid extension.  As a consequence, we observe that Crowley et al.'s formula for the rank function of the minimal multisymmetric lift affords a solution of Problem~\ref{prob:rank-generous} in the special case of U-matroids arising from subspace arrangements.

\subsection{U-matroids and poset matroids} \label{sec:posetmatroid}

\begin{defn}\cite[\S11]{Barnabei}
A \defterm{poset matroid} is a triple $U=(E,\D,\rho)$, where $E$ is a finite set, $\D\subseteq 2^E$ is an accessible lattice, and $\rho$ is a function $\D\to\Nn$ that satisfies calibration~\eqref{calibration}, monotonicity~\eqref{monotonicity}, unit-increase~\eqref{unit-increase}, and the following \defterm{local chain property}.  Let $A\in\D$ and $b_1,b_2\in E$, and write $A_1=A\cup b_1$, $A_2=A\cup b_2$, $A_{12}=A\cup b_1\cup b_2$.  Suppose that $A_{12}\in\D$ and $\rho(A_{12})>\rho(A)$.  The local chain property is then that for at least one $j\in\{1,2\}$, we have $A_j\in\D$ and $\rho(A_j)=\rho(A)+1$.
\end{defn}

Poset matroid rank functions are submodular \cite[Prop.~11.7]{Barnabei}, so every poset matroid is a U-matroid.
Barnabei et al. defined the basis system of a poset matroid $U=(E,\D,\rho)$ by
\begin{equation} \label{P-basis-from-rank}
\dot\B=\dot\B(U)=\{\text{maximal elements $B\in\D$ such that $|B|=\rho(B)$}\}.
\end{equation}
The elements of $\dot\B$ were regarded in \cite{Barnabei} as the bases of $U$, and as we will see, the two notions coincide for poset matroids (although the definition of $\dot\B$ is not well-behaved in general for U-matroids).
All elements of $\dot\B(U)$ have equal cardinality~\cite[Thm.~4.1]{Barnabei}, so in particular $\dot\B$ is an antichain in $\D$.

\begin{thm} \label{thm:poset matroids-as-U-matroids}
A U-matroid $U=(E,\D,\rho)$ is a poset matroid if and only if every basis of $U$ is an element of $\D$.
\end{thm}

\begin{proof}
($\Longrightarrow$) Suppose that $\rho$ is a poset matroid rank function.  Let $B\in\B(U)$; we must show $B\in\D$.  Let $\0=A_0\coveredby\cdots\coveredby A_n=E$ be a maximal chain in $\D$ such that $B=B_\rho(A)$.  Write $A_i=\{a_1,\dots,a_i\}$.  Suppose that there are consecutive indices $i,i+1,i+2$ such that $\rho(A_i)=\rho(A_{i+1})=\rho(A_{i+2})+1$.  Then the local chain property implies that $A'_{i+1}=A_i\cup\{a_{i+2}\}$ belongs to $\D$ and has rank $\rho(A_i)+1$. Replacing $A_{i+1}$ with $A'_{i+1}$ produces a chain $A'$ such that $B_\rho(A')=B$.  Now replace $A$ with $A'$ and iterate this process until it stops.  The result is a chain in which $\rho(A_i)=\min(i,\rho(E))$.  In particular, $B=A_r\in\D$.
\medskip

($\Longleftarrow$) Suppose that $\B(U)\subseteq\D$.  We must show that $\rho$ satisfies the local chain property.  Let $A\in\D$ and $b_1,b_2\in E$; write $A_1=A\cup b_1$, $A_2=A\cup b_2$, $A_{12}=A\cup b_1\cup b_2$.  Suppose that $A_{12}\in\D$ and $\rho(A_{12})>\rho(A)$; by unit-increase $\rho(A_{12})\in\{\rho(A)+1,\rho(A)+2\}$.  If $\rho(A_{12})=\rho(A)+2$ then the local chain property follows from accessibility and unit-increase, and if $A_1$ and $A_2$ both belong to $\D$ then it follows from submodularity of $\rho$.  Thus, if the local chain property fails, then the only remaining possibility is (w.l.o.g.)\ that
$A_1\in\D$, $A_2\notin\D$, and $\rho(A_1)=\rho(A)$.
Now extend $A\coveredby A_1\coveredby A_{12}$ to a maximal chain in $\D$, which gives rise to a basis $B$.  Evidently $b_1\notin B$ and $b_2\in B$.  But by hypothesis $B\in\D$, so $(A\cup B)\cap A_{12}=A_2\in\D$, a contradiction.
\end{proof}

\begin{cor}
Let $U=(E,D,\rho)$ be a poset matroid.  Then $\dot\B(U)=\B(U)$.
\end{cor}
\begin{proof}
The construction in the first part of the proof of Theorem~\ref{thm:poset matroids-as-U-matroids} shows that $\B(U)\subseteq\dot\B(U)$.  On the other hand, if $B\in\dot\B(U)$, then for any maximal chain in $\D$ containing $B$, the corresponding U-matroid basis is just $B$ (and at least one such chain exists).  Therefore, equality holds.
\end{proof}

The class of U-matroids strictly contains the class of poset matroids.  For example, the U-matroid of Example~\ref{ex:not-comb-scheme} is not a poset matroid, because the basis 134 does not belong to the characteristic lattice.  Moreover, the data $(A,b_1,b_2)=(14,2,3)$ violates the local chain property.

\subsection{Subspace arrangements} \label{sec:arr}

Let $\fld$ be a field and $V$ a vector space over~$\fld$ of dimension~$d$.  Write $\Gr^c(n)$ for the Grassmannian of codimension-$c$ vector subspaces of $\fld^n$.
For nonnegative integers $\cc=(c_1,\dots,c_m)$, a \defterm{$\cc$-arrangement} is a list $\XX=(X_1,\dots,X_m)$ of vector subspaces of~$V$ with $\codim X_i=c_i$ for all~$i$.  We define
\begin{align*}
\Gr(V,\cc)
&= \{\text{$\cc$-arrangements $\XX$ in $V$}\}\\
&\isom \Gr^{c_1}(d)\x\cdots\x\Gr^{c_m}(d).
\end{align*}
Let $\cc=(c_1,\dots,c_m)$, $\bb=(b_1,\dots,b_m)$ such that $c_i\geq b_i$ for all $i$ (for short, $\cc\geq\bb$).  For $\XX\in\Gr(V,\cc)$ and $\YY\in\Gr(V,\bb)$, we write $\YY\supseteq\XX$ to mean that $Y_i\supseteq X_i$ for every $i$.  We define
\begin{align*}
\Gr_\XX(V,\bb)
&= \{\YY\in\Gr(V,\bb)\st \YY\supseteq\XX\}\\
&\isom \{(Y_1/X_1,\dots,Y_m/X_m)\st Y_i\in\Gr^{b_i}(V),\ Y_i\supset X_i\}\\
&= \{(Y_1/X_1,\dots,Y_m/X_m)\st Y_i/X_i\in\Gr^{b_i}(V/X_i)\}\\
&\isom \Gr^{b_1}(d-c_1)\x\cdots\x\Gr^{b_m}(d-c_m).
\end{align*}
Also, we write $\codim\cap\XX=\codim(X_1\cap\cdots\cap X_m)$.

Let $\XX\in\Gr(V,\cc)$, and let $\D_\cc$ be the distributive lattice $[0,c_1]\x\cdots\x[0,c_m]$.
Barnabei et al.\ \cite[Thm.~12.1]{Barnabei} proved that $([n],\D_\cc,\rho_\XX)$ is a poset matroid, where $n=\sum c_i$ and $\rho_\XX:\D_\cc\to\Nn$ is defined by
\begin{equation} \label{rank-from-subspace-arrangement}
\rho_\XX(b_1,\dots,b_m)
= \max\{\codim\cap\YY\st \YY\in\Gr_\XX(V,\bb)\}. 
\end{equation}
Formula~\eqref{rank-from-subspace-arrangement} is the one used in~\cite[Thm.~12.1]{Barnabei}.

Henceforth, we assume that $\fld$ is infinite, so that we can work with generic sets in the Zariski topology on $\Gr(V,\cc)$.
Since $\Gr_\XX(V,\bb)$ is an irreducible variety in which each condition $\codim\cap\YY\leq k$ is Zariski-closed (by the vanishing of appropriate Pl\"ucker coordinates), it follows that~\eqref{rank-from-subspace-arrangement} is equivalent to
\begin{equation} \label{rank-generic}
\rho_\XX(b_1,\dots,b_m) = \codim\cap\YY, \quad \text{ for a generic }\YY\in\Gr_\XX(V,\bb).
\end{equation}

In particular, if $\XX$ is a hyperplane arrangement, then $\rho_\XX$ is the rank function of the matroid represented by normals to its hyperplanes, as usual.

\begin{defn} \label{defn:representable}
A U-matroid $(E,\D,\rho)$ is \defterm{representable} (over $\fld$) if it arises from a subspace arrangement by this construction.
(In particular, $\D$ must be a product of chains, and the rank of every join-irreducible element must equal its height in $\D$.)
\end{defn}

The \textit{rank function}\footnote{Note that $\rk$ is \textit{not} the rank function of a submodular system in the sense of Definition~\ref{def:submodsyst}.  We will always use the symbol $\rho$ for the rank function of a submodular system in the sense of Fujishige~\cite{Fujishige}, and $\rk$ for the rank function of a subspace arrangement in the sense of Bj\"orner~\cite{Bjorner-subspace}.}
of a $\cc$-arrangement $\XX=(X_1,\dots,X_m)$ is the function $\rk:2^{[m]}\to\Nn$ given by
\begin{equation}\label{rank-arrangement-polymatroid}
\rk_\XX(A)=\codim\left(\bigcap_{a\in A}X_a\right).
\end{equation}
In~\cite[\S4.1]{Bjorner-subspace}, the rank function is defined on the lattice of intersections of the $X_i$, rather than on the full Boolean lattice, but the two are easily seen to be equivalent.  The function $\rk_\XX$ is a polymatroid rank function~\cite[Example~1.1]{Crowley}.

Let $P$ be an integral polymatroid $P$ on ground set $[m]$. Crowley et al. \cite[\S2.1]{Crowley} construct a matroid $M=\tilde P$
called its \defterm{minimal multisymmetric lift}, as follows.  Let $\tilde E_1,\dots,\tilde E_m$ be disjoint sets with $|E_i|=\rk_P(i)$ and $\tilde E=\tilde E_1\sqcup\cdots\sqcup\tilde E_n$. Define a projection map $\pi:\tilde E\to[n]$ by $\pi^{-1}(i)=
\tilde E_i$.  Then define for $S\subseteq\tilde E$
\begin{equation} \label{rank-lift}
\rk_M(S)=\min\{\rk_P(A)+|S\sm\pi^{-1}(A)|\st A\subseteq E\}
\end{equation}
By \cite[Thm.~2.9]{Crowley}, $\rk_M$ is a matroid rank function, and it is \textit{multisymmetric} in the sense that $\rk_M(S)$ depends only on the tuple $(|S\cap\tilde E_1|,\dots,|S\cap\tilde E_m|)$.  If $\rk_P=\rk_\XX$ for a subspace arrangement $\XX=(X_1,\dots,X_m)$, then $M$ is the matroid associated with any hyperplane arrangement obtained by replacing each $X_i$ by $c_i$ generic hyperplanes, each containing $X_i$ \cite[Example~2.10]{Crowley}.

\begin{prop}
Let $\cc=(c_1,\dots,c_m)$ and let $\XX$ be a $\cc$-arrangement.  The following objects all contain the same information:
\begin{enumerate}
\item The intersection lattice and rank function in the sense of Bj\"orner \cite{Bjorner-subspace};
\item The polymatroid rank function $\rk_\XX$ defined by~\eqref{rank-arrangement-polymatroid};
\item The minimal multisymmetric lift $\rk_M$ defined by~\eqref{rank-lift};
\item The poset matroid (or U-matroid) rank function $\rho_\XX$ defined by~\eqref{rank-generic}.
\end{enumerate}
\end{prop}

\begin{proof}
As we have observed, the function $\rk_\XX$ is equivalent to the intersection lattice and rank function in Bj\"orner's sense.
The minimal multisymmetric lift $M=\tilde P$ depends only on $\rk_P$, and conversely one can show (using the fact that $\rk_P$ is not an arbitrary poset matroid, but arises from a subspace arrangement) that \eqref{rank-lift} specializes to $\rk_M(\bigcup_{a\in A}\tilde E_a)=\rk_P(A)$ for all $A\subseteq[m]$.

The U-matroid of an arrangement determines its polymatroid, since $\rk_\XX$ is just the restriction of $\rho_\XX$ to the Boolean sublattice $\{0,c_1\}\x\cdots\x\{0,c_m\}\subseteq\D_\cc$.
On the other hand, by \cite[Ex.~2.10]{Crowley}, $\rho_\XX$ is just the restriction of $\rk_M$ to $\D_\cc$, regarding the latter as a sublattice of $2^{\tilde E}$.  (Indeed, this is a lattice restriction in the sense of Definition~\ref{defn:lattice-restriction}.)
More explicitly,
\[\rho_\XX(b_1,\dots,b_m)=\rk_M(A_1\sqcup\cdots\sqcup A_m)\]
for any sets $A_1\in\binom{\tilde E_1}{b_1},\dots,A_m\in\binom{\tilde E_m}{b_m}$.
\end{proof}

\subsection{Generous extensions revisited} \label{sec:en-again}

If the U-matroid of a subspace arrangement contains the same information as these other constructions, how is it useful?  One answer to this question is that, under suitable conditions, generous atom extension of U-matroids corresponds to a small intermediate step in the geometric construction of the minimal multisymmetric lift, so that the generous matroid extension is precisely the minimal multisymmetric lift itself.  We begin with some examples of generous extensions of poset matroids.

\begin{example}
Let $\XX$ be the $3$-equal subspace arrangement in $\Rr^4$, consisting of the four planes defined by the equations $x_i=x_j=x_k$, where $1\leq i<j<k\leq 4$.  In the corresponding poset matroid, the characteristic lattice $\D$ is a product of four chains of length two; equivalently, the characteristic poset has ground set $[8]$ and relations $1<2,3<4,5<6,7<8$. The generous extension of the corresponding poset matroid is the uniform matroid $U_3(8)$.  This matroid is representable by a hyperplane arrangement consisting of eight generic hyperplanes.  Geometrically, each of the four planes $P\in\XX$ has been replaced with two generic hyperplanes whose intersection is $P$.
\end{example}

\begin{thm} \label{generous-geometry}
Let $\XX=(X_1,\dots,X_m)\in\Gr(V,\cc)$ as above.  Assume $c_m\geq2$.  Let $\cc'=(c'_1,\dots,c'_{m+1})=(c_1,\dots,c_{m-1},c_m-1,1)$, and let
$\XX'=(X_1,\dots,X_{m-1},X'_m,H)\in\Gr(V,\cc')$, where $X'_m,H$ are chosen generically so that $\codim X'_m=c_m-1$, $\codim H=1$, and $X'_m\cap H=X_m$.  Let $\rho=\rho_{\XX}$ and $\rho'=\rho_{\XX'}$, as in~\eqref{rank-from-subspace-arrangement}.  Then $\rho'$ is the generous atom extension of $\rho$ by the atom $n$.
\end{thm}

We first make some observations about the poset matroids $([n],\D,\rho)$ and $([n],\D',\rho')$ arising from~$\XX$ and~$\XX'$.   We can identify the lattices $\D$ and $\D'=\D[n]$ with $[0,c_1]\x\cdots\x[0,c_m]$ and $[0,c_1]\x\cdots\x[0,c'_{m+1}]$ respectively.  The poset $\P'=\Irr(\D')$ is obtained from $\P=\Irr(\D)$ by deleting all the relations of the form $b<n$, so that $n$ has no relations in $\P'$.  Any generous atom extension of a poset matroid over an element $e\in E$ which is maximal in $\P$ may be realized in this way.  (The generous atom extension by an element $e\in E$ which is not maximal in $\P$ is defined over a distributive lattice $\D[e]$ that is not a product of chains, and hence does not correspond to a poset matroid.)

\begin{proof}[Proof of Theorem~\ref{generous-geometry}]
Let $\bb'=(b_1,\dots,b_{m+1})\in\D'$, so that $\bb=(b_1,\dots,b_m)\in\D$.  We must show that
\begin{numcases}{\rho'(\bb')=}
\rho(b_1,\dots,b_{m-1},b_m+b_{m+1}) & \text{ if $b_m=c_m-1$ or $b_{m+1}=0$}, \label{genhyp:1}\\
\rho(\bb) & \text{ if $b_m<c_m-1$, $b_{m+1}=1$, and $\rho(\bb)=\rho(b_1,\dots,b_{m-1},c_m)$}, \label{genhyp:2}\\
\rho(\bb)+1 & \text{ if $b_m<c_m-1$, $b_{m+1}=1$, and $\rho(\bb)<\rho(b_1,\dots,b_{m-1},c_m)$}. \label{genhyp:3}
\end{numcases}
which is equivalent to the formula of Definition~\ref{defn:generous-atom}.  By~\eqref{rank-generic}, we know that
\begin{align*}
\rho(\bb) &= \codim(Y_1\cap\cdots\cap Y_m),\\
\rho'(\bb') &= \codim(Y'_1\cap\cdots\cap Y'_{m+1}),
\end{align*}
for generically chosen $Y_i\in\Gr_{X_i}(V,b_i)$ and $Y'_i\in\Gr_{X'_i}(V,b_i)$.  For $1\leq i\leq m-1$, we know that $X_i=X'_i$ and thus may assume $Y_i=Y'_i$.  Moreover, we may assume $Y_m=Y'_m$, since $X'_m$ was chosen generically from all spaces containing $X_m$ as a hyperplane, so the requirement $Y'_m\supseteq X'_m$ is no more restrictive than the requirement $Y_m\supseteq X_m$.  Hence we may rewrite the previous formulas as
\begin{equation} \label{rhobb}
\begin{aligned}
\rho(\bb) &= \codim(Y_1\cap\cdots\cap Y_m), \\
\rho'(\bb') &= \codim((Y_1\cap\cdots\cap Y_m)\cap Y'_{m+1}).
\end{aligned}
\end{equation}

We now consider various cases.
\medskip

\noindent\underline{\textbf{Case 1: $b_{m+1}=0$.}}
Then $Y'_{m+1}=V$, and by \eqref{rhobb} we have $\rho(\bb)=\rho'(\bb')$, which is consistent with~\eqref{genhyp:1}.
\medskip

\noindent\underline{\textbf{Case 2: $b_{m+1}=1$ and $b_m=c_m-1$.}}
Then $Y'_{m+1}=H$ and the desired equation~\eqref{genhyp:1} reduces to $\rho'(\bb')=\rho(b_1,\dots,b_{m-1},c_m)$.
In the formula for $\rho(\bb')$ in~\eqref{rhobb},
the space $Y_m$ is a generic element of $\Gr_{X_m}(V,b_m)=\Gr_{X_m}(V,c_m-1)$ and 
$H$ is a generic element of $\Gr_{X_m}(V,1)$, so $Y_m\cap H$ contains $X_m$ and has codimension $c_m$.  But then in fact $Y_m\cap H=X_m$, so~\eqref{rhobb} becomes
\[
\rho'(\bb') = \codim(Y_1\cap\cdots\cap Y_{m-1}\cap X_m)=\rho(b_1,\dots,b_{m-1},c_m)
\]
as desired.
\medskip

\noindent\underline{\textbf{Case 3: $b_{m+1}=1$ and $b_m<c_m-1$.}}
Again, $Y'_{m+1}=H$.  Since $H$ is generic, \eqref{rhobb} says that
\begin{equation} \label{rhobb:new}
\rho'(\bb')=\begin{cases}
\rho(\bb) &\text{ if } H\supseteq Y_1\cap\cdots\cap Y_m,\\
\rho(\bb)+1 &\text{ otherwise.}
\end{cases}
\end{equation}
Meanwhile,
\begin{align*}
H&\supseteq Y_1\cap\cdots\cap Y_m \text{ for generic } \YY=(Y_1,\dots,Y_m)\in\Gr_{\XX}(V,\bb)\\
&\iff X_m\supseteq Y_1\cap\cdots\cap Y_m \text{ for generic } \YY\in\Gr_{\XX}(V,\bb)\\
&\iff X_m\supseteq Y_1\cap\cdots\cap Y_m \text{ for all } \YY\in\Gr_{\XX}(V,\bb)
\intertext{(since the containment condition is Zariski closed)}
&\iff X_m\cap(Y_1\cap\cdots\cap Y_m)=Y_1\cap\cdots\cap Y_m \text{ for all } \YY\in\Gr_{\XX}(V,\bb)\\
&\iff Y_1\cap\cdots\cap Y_{m-1}\cap X_m=Y_1\cap\cdots\cap Y_m \text{ for all } \YY\in\Gr_{\XX}(V,\bb)\\
&\iff \codim(Y_1\cap\cdots\cap Y_{m-1}\cap X_m)=\codim(Y_1\cap\cdots\cap Y_m) \text{ for all } \YY\in\Gr_{\XX}(V,\bb)\\
&\iff \rho(b_1,\dots,b_{m-1},c_m)=\rho(b_1,\dots,b_m).
\end{align*}
Combining this observation with~\eqref{rhobb:new} gives
\begin{equation} \label{rhobb:newer}
\rho'(\bb')=\begin{cases}
\rho(\bb) &\text{ if } \rho(b_1,\dots,b_{m-1},c_m)=\rho(b_1,\dots,b_m),\\
\rho(\bb)+1 &\text{ if } \rho(b_1,\dots,b_{m-1},c_m)>\rho(b_1,\dots,b_m)
\end{cases}
\end{equation}
which establishes cases \eqref{genhyp:2} and \eqref{genhyp:3}.
\end{proof}

\begin{cor}
If $\XX$ is a subspace arrangement with U-matroid $U$, then the minimal multisymmetric lift of $U$ (in the sense of Crowley et al.~\cite{Crowley}) is precisely the generous matroid extension.  In particular, the rank function of the generous matroid extension is given in closed form by~\eqref{rank-lift}, solving Problem~\ref{prob:rank-generous} for the special case of representable U-matroids.
\end{cor}

\begin{cor}
If $\fld$ is an infinite field and $U$ is a $\fld$-representable U-matroid, then so is every iterated generous atom extension of $U$; in particular, so is the generous matroid extension.
\end{cor}

The requirement $|\fld|=\infty$ cannot be dropped, for the following reason.

\begin{example} \label{ex:field-matters}
The stalactite of Example~\ref{ex:stalactite} is a poset matroid realizable over any field~$\fld$, arising from a $(1,1,2)$-arrangement in $\fld^2$ consisting of two distinct lines and the origin. However, its generous matroid extension is $U_2(4)$, which is representable over $\Rr$ but not over $\Ff_2$.
\end{example}

\begin{qn}
More generally, is every generous extension of a $\fld$-representable poset matroid $U$ to a product of chains (i) a poset matroid; (ii) if so, representable?  (If the answer to both questions is yes, then we could hope to understand the relationship between a subspace arrangement representing $U$ and an arrangement giving rise to the generous extension, which should be some generalization of Theorem~\ref{generous-geometry}.)
\end{qn}

\begin{qn}
Related question: Under what conditions is the restriction of a representable poset matroid to a lattice which is a product of chains representable?
This question appears to be difficult in light of the following examples.
\end{qn}

\begin{example}
Let $\XX$ consist of two 2-planes in $\fld$-space that meet in a line, and let $U=(E,\D,\rho)$ be the corresponding poset matroid.  Identifying $\D$ with $[0,2]\x[0,2]$, we have $\rho(0,0)=0$ and $\rho(2,2)=3$.  In particular, we cannot obtain a representable poset matroid by restricting to any maximal subchain.
\end{example}

More generally, let $\XX=(X_1,\dots,X_m)$ be a $\cc$-arrangement in $V$ and $U=(E,\D,\rho)$ its corresponding poset matroid.  As in \S\ref{sec:arr}, the lattice $\D$ may be identified with $[0,c_i]\x\cdots\x[0,c_m]$, and $\rho=\rho_\XX$ is defined by~\eqref{rank-from-subspace-arrangement}.  Let $C$ be any maximal chain in $[0,c_1]\x[0,c_2]$, say $C=\{(p_i,q_i)\st0\leq i\leq c_1+c_2\}$ with $\rho(p_i,q_i)=i$.  Then $\D'=C\x[0,c_3]\x\cdots\x[0,c_m]\isom[0,c_1+c_2]\x[0,c_3]\x\cdots\x[0,c_m]$ is an accessible distributive sublattice of $\D$.  It is easy to show that $U'=(E,\D',\rho|_{\D'})$ is representable only if $X_1,X_2$ intersect transversely, i.e., if $\rho(c_1,c_2,0,\dots,0)=c_1+c_2$.  One might hope that in this case, the subspace arrangement represented by $U'$ is obtained from $\XX$ by replacing the pair $X_1,X_2$ with their transverse intersection $X_1\cap X_2$.  However, this is not the case, as the following examples illustrate.

\begin{example}
Let $V=\fld^4$, with basis $\{\ee_1,\dots,\ee_4\}$, and let $\XX$ be the arrangement consisting of the subspaces
\[X_1=\langle \ee_1,\ee_2\rangle^\perp, \quad X_2=\langle \ee_3,\ee_4\rangle^\perp,\quad X_3=\langle \ee_1,\ee_3\rangle^\perp.\]
In particular, $X_1$ and $X_2$ intersect transversely.
The corresponding poset matroid has characteristic lattice $[0,2]\x[0,2]\x[0,2]$ and rank function
\[\rho(\bb) = \begin{cases} 3 & \text{ if } \bb=(2,0,2), \\ \min(a+b+c,4) & \text{ otherwise.}\end{cases}\]
If $C$ is a maximal chain in $[0,2]\x[0,2]$ that contains the element $(1,1)$, then $\rho|_{C\x[0,2]}$ is represented by the arrangement $(X_1\cap X_2,X_3)$, as one would hope.  On the other hand, if $C$ is any other maximal chain, then $\rho|_{C\x[0,2]}$ is not representable.
\end{example}

\begin{example}
Let $V=\fld^6$, with basis $\{\ee_1,\dots,\ee_6\}$, and let $\XX$ be the arrangement consisting of the subspaces
\[
X_1 = \langle \ee_1,\ee_2,\ee_3\rangle^\perp, \quad
X_2 = \langle \ee_4,\ee_5,\ee_6\rangle^\perp, \quad
X_3 = \langle \ee_1,\ee_2,\ee_6\rangle^\perp, \quad
X_4 = \langle \ee_3,\ee_4,\ee_5\rangle^\perp.
\]
Let $U$ be the corresponding poset matroid, with $\D=[0,3]^4$ and rank function $\rho$.
Again $X_1$ and $X_2$ meet transversely, but now there exists \textit{no} maximal chain $C\subseteq[0,3]\x[0,3]$ such that $\D_C=C\x[0,3]\x[0,3]$ represents the arrangement $\XX'=(X_1\cap X_2,X_3,X_4)$.  Sixteen of the twenty possibilities for $C$ give rise to non-representable matroids.  The remaining four are representable by an arrangement consisting of the zero space and two 3-spaces that intersect in a line.  However, we do not see how to obtain such an arrangement naturally from $\XX$ by first replacing $X_1,X_2$ with $X_1\cap X_2=0$, because it then seems impossible to retain the information carried by $\XX$ about the numbers $\codim(X_i\cap X_j)$ for $i=1,2$ and $j=3,4$.
\end{example}

\section{Further questions} \label{sec:further}

\begin{enumerate}
\item In keeping with Rota's program, one could generalize other matroidal concepts to U-matroids: loops and coloops, deletion/contraction, basis exchange properties, circuits, the Tutte polynomial, internal and external activities, the broken-circuit complex, etc.  In matroid theory, the last two rely on a choice of ordering of the ground set, but exhibit properties independent of the choice of ordering.  Here, results such as Theorem~\ref{thm:basis-theorem} suggest that the right orderings to consider are the linear extensions of~$\P$. Circuits may be a particularly important axiomatization in light of their central role in the Orlik-Solomon algebra of a matroid, as described in the next problem.

\item 
Develop a theory of \textit{unbounded oriented matroids}, with a view toward the important problem of understanding the topology of a subspace arrangement complements in terms of its combinatorics.  For hyperplane arrangements over~$\Cc$, the cohomology ring was famously described by Orlik and Solomon~\cite{OS}, and for real subspace arrangements $\XX\subseteq\Rr^d$, Goresky and McPherson~\cite{GM} computed the homology using the intersection lattice together with the rank function.   Feichtner and Ziegler \cite{FZ} extended the Orlik--Solomon construction to complex subspace arrangements with geometric intersection lattice (a very strong condition), and gave a conjecture, later proven by de~Longueville and Schultz~\cite{deLSch} for the more general class of real arrangements in which the codimensions of any two related elements of the intersection lattice differ by 2 or more.  The construction of \cite{deLSch} enriches the combinatorics of the intersection lattice by choosing an orientation of each subspace, motivating the question of defining orientations for unbounded matroids (particularly for poset matroids).

\item In analogy with the definition of representability of a P-matroid, say that a U-matroid $M=([n],\D,\rho)$ is \defterm{represented} by a collection of vectors $\{v_1,\dots,v_n\}$ if $\rho(A)=\dim\textup{Span}\{v_i\st i\in A\}$ for all $A\in\D$. Note that this is weaker than the notion of representability of P-matroids considered in Section~\ref{sec:PMatroid}.
Representable matroids give a stratification of the Grassmannian, wherein an element of the Grassmannian (realized as a matrix $A$) is associated to the matroid whose bases are the nonvanishing Pl\"ucker coordinates of $A$. Using the above notion of representability, we may consider the U-matroid analogue of this picture by associating a point of the Grassmannian to the U-matroid whose bases are the nonvanishing Pl\"ucker coordinates which correspond to elements of $\D$. This gives a U-matroid stratification of the Grassmannian which is a coarsening of the matroid stratification. One may also consider the U-analogue of positroids and the totally nonnegative Grassmannian by considering a point in the Grassmannian to be $\D$-nonnegative if it may be represented by a matrix all of whose nonvanishing maximal minors corresponding to elements of $\D$ are positive.

\item How do geometric invariants of the base polyhedron correspond to combinatorial invariants of the rank function?  For example, if $M$ is a matroid on $n$ elements, then $\dim\BBB(M)$ equals $n$ minus the number of connected components of~$M$.  What is the analogous combinatorial description of $\dim\conv(\B(U))$ for a U-matroid $U$?

\item Recall that $\BBB(\hat A)=\BBB(A)\cap[0,1]^n$ (Theorem~\ref{generous-polytopes}).  For each basis $B$ of $\hat A$, let $f(B)$ be the smallest number $k$ such that the characteristic vector of $B$ belongs to the $k$-skeleton of $\BBB(A)$.  What does the number $f(B)$ say about $B$?  What does the generating function $\sum_B x^{f(B)}$ say about $A$?  (Note that these questions have trivial answers if $A$ is itself a matroid.)

\item Submodular systems are studied for their applications in optimization \cite{Fujishige}.  Do U-matroids arise from concrete optimization problems?
\end{enumerate}

\bibliographystyle{abbrv}
\bibliography{biblio}
\end{document}